\documentclass[11pt,oneside]{amsart} % remove `oneside' in final version
\usepackage{amscd,amssymb,amsxtra}
\usepackage[mathscr]{eucal}% \usepackage{mathrsfs}  % \mathscr{A}
\usepackage{comment}
\usepackage{color}
\usepackage{enumerate}
\makeatletter
\let\@secnumfont\bfseries
\def\section{\@startsection{section}{1}%
  \z@{4\linespacing\@plus\linespacing}{\linespacing}%
  {\bfseries\centering}}
\def\introsection{\@startsection{section}{1}%
  \z@{3\linespacing\@plus\linespacing}{\linespacing}%
  {\bfseries\centering}}
\def\subsection{\@startsection{subsection}{2}%
   \z@{1.25\linespacing\@plus.7\linespacing}{.5\linespacing}%
   {\normalfont\bfseries}}
\def\subsectionsinline{\def\subsection{\@startsection{subsection}{2}%
  \z@{1\linespacing\@plus.7\linespacing}{-.5em}%
  {\normalfont\bfseries}}}

\makeatother

\theoremstyle{definition}
\newtheorem{definition}[equation]{Definition}
\newtheorem{example}[equation]{Example}
\newtheorem{problem}[equation]{Problem}

\newtheorem*{definition*}{Definition}
\newtheorem*{example*}{Example}
\newtheorem*{problem*}{Problem}
\newtheorem*{exercise*}{Exercise}
\newtheorem*{question*}{Question}
\newtheorem*{construction*}{Construction}

\theoremstyle{remark}

\newtheorem{remark}[equation]{Remark}

\newtheorem*{note*}{Note}
\newtheorem*{notation*}{Notation}
\newtheorem*{remark*}{Remark}
\newtheorem*{data*}{Data}

\theoremstyle{plain}
\newtheorem{theorem}[equation]{Theorem}
\newtheorem{corollary}[equation]{Corollary}
\newtheorem{lemma}[equation]{Lemma}
\newtheorem{proposition}[equation]{Proposition}

\newtheorem*{theorem*}{Theorem}
\newtheorem*{corollary*}{Corollary}
\newtheorem*{lemma*}{Lemma}
\newtheorem*{proposition*}{Proposition}
\newtheorem*{conjecture*}{Conjecture}
\newtheorem*{claim*}{Claim}
\newtheorem*{proposal*}{Proposal}
\newtheorem*{conclusion*}{Conclusion}
\newtheorem*{hypothesis*}{Hypothesis}

\numberwithin{equation}{section}

\definecolor{refkey}{rgb}{0,.6,.4}

\renewcommand{\:}{\colon}
\newcommand{\CC}{{\mathbb C}}

\newcommand{\EE}{\mathbb E}

\DeclareMathOperator{\End}{End}

\DeclareMathOperator{\Hom}{Hom}
\DeclareMathOperator{\id}{id}
\DeclareMathOperator{\Map}{Map}

\newcommand{\PP}{{\mathbb P}}
\DeclareMathOperator{\pt}{pt}

\newcommand{\RP}{{\mathbb R\mathbb P}}
\newcommand{\RR}{{\mathbb R}}
\newcommand{\TT}{\mathbb T}

\newcommand{\ZZ}{{\mathbb Z}}

\newcommand{\chiup}{\raise.5ex\hbox{$\chi$}}
\newcommand{\cir}{S^1}

\newcommand{\inv}{^{-1}}
\newcommand{\mstrut}{^{\vphantom{1*\prime y\vee M}}}

\newcommand{\res}[1]{\negmedspace\bigm|\mstrut_{#1}}

\newcommand{\temsquare}{\raise3.5pt\hbox{\boxed{ }}}

\newcommand{\zmod}[1]{\ZZ/#1\ZZ}

\newcommand{\zt}{\zmod2}

\usepackage[all,2cell,dvips]{xy}\renewcommand{\cir}{\ensuremath{S^1}}
\usepackage[colorlinks,backref,citecolor=refkey]{hyperref}

\newcommand{\vect}{\mathbf{Vect}}

\DeclareMathOperator*{\colim}{colim}
\DeclareMathOperator{\Ad}{Ad}
\DeclareMathOperator{\Koss}{Kos}
\DeclareMathOperator{\Lie}{Lie}
\DeclareMathOperator{\Open}{Open}
\DeclareMathOperator{\Sing}{Sing}
\DeclareMathOperator{\Sym}{Sym}
\DeclareMathOperator{\ho}{ho}
\DeclareMathOperator{\hypercover}{\mathbf{Hypercov}}
\DeclareMathOperator{\manifolds}{\mathbf{Man}}

\DeclareMathOperator{\presheaves}{\mathbf{Pre}}
\DeclareMathOperator{\sheaves}{\mathbf{Sh}}
\DeclareMathOperator{\spre}{\mathbf{sPre}}
\newcommand{\BNGt}{B_{\nabla }^{\textnormal{triv}}G}
\newcommand{\BNG}{B^{}_{\nabla }G}
\newcommand{\Bqt}{(X_G)_\nabla }
\newcommand{\ENG}{E^{}_{\nabla }G}
\newcommand{\Grp}{\mathbf{Groupoid}}
\newcommand{\Kos}{\Koss ^{\bullet }}

\newcommand{\Man}{\mathbf{Man}}
\newcommand{\OtV}{\Omega ^1\otimes V}
\newcommand{\Otg}{\Omega ^1\otimes \mathfrak{g}}
\newcommand{\SSet}{\Set_\Delta }

\newcommand{\Set}{\mathbf{Set}}
\newcommand{\Spc}{\mathbf{Space}}
\newcommand{\TW}{T\mstrut _{W}}
\newcommand{\Tens}{\otimes }
\newcommand{\Tu}{\Theta^{\textnormal{univ}}}
\newcommand{\Vect}{\mathbf{Vect}}
\newcommand{\abelian}{\mathbf{Ab}}
\newcommand{\ass}{\mathbf{a}}
\newcommand{\bX}{X_{\bullet }}
\newcommand{\bas}{_{\textnormal{bas}}}

\newcommand{\bx}{\bar\xi }
\newcommand{\cat}[1]{\mathcal{#1}}
\newcommand{\chain}[1]{\mathbf{Chain}_{#1}}
\newcommand{\df}[1]{\Omega ^{\bullet }(#1)}

\newcommand{\hx}{\hat\xi }
\newcommand{\omo}{\boldsymbol{\omega }^1}
\newcommand{\omq}{\boldsymbol{\omega }^q}
\newcommand{\op}{^{\textnormal op}}
\newcommand{\sFX}{\sF_X}
\newcommand{\sF}{\mathcal{F}}
\newcommand{\sG}{\mathcal{G}}
\newcommand{\sL}{\mathcal{L}}
\newcommand{\sP}{\mathcal{P}}
\newcommand{\sS}{\mathcal{S}}
\newcommand{\sU}{\mathcal{U}}
\newcommand{\sheaf}[1]{\mathcal{#1}}

\newcommand{\stF}{\widetilde{\sF}}
\newcommand{\tP}{\tilde{P} }
\newcommand{\tT}{\tilde{\Theta } }
\newcommand{\ta}{\tilde\alpha }
\newcommand{\tb}{\tilde\beta }
\newcommand{\tf}{\tilde\phi }
\newcommand{\tih}{\tilde{h}}
\newcommand{\tla}{\tilde{\lambda }}
\newcommand{\tl}{\tilde{\ell }}
\newcommand{\tp}{\tilde\pi }
\newcommand{\ts}{\tilde{s}}
\newcommand{\tv}{\tilde{v}}
\newcommand{\tw}{\tilde\omega }
\newcommand{\weq}{\mathcal W}

\renewcommand{\AA}{\mathbb{A}}

\ifx\undefined\slot
\newcommand{\slot}{\,-\,}
\fi

\begin{document}

\abovedisplayskip18pt plus4.5pt minus9pt
\belowdisplayskip \abovedisplayskip
\abovedisplayshortskip0pt plus4.5pt
\belowdisplayshortskip10.5pt plus4.5pt minus6pt
\baselineskip=15 truept
\marginparwidth=55pt

\renewcommand{\labelenumi}{\textnormal{(\roman{enumi})}}
\setcounter{tocdepth}{1}

%**end of header

% lasteq@153
% lastsec@ 11
% lastthm@ 78
% lastfig@000

 \title[Chern-Weil Forms and Abstract Homotopy Theory]{Chern-Weil Forms and Abstract Homotopy Theory} %% replace \today with short title in final version
 \author[D. S. Freed]{Daniel S.~Freed}
 \thanks{The work of D.S.F. is supported by the National Science Foundation
under grants DMS-0603964, DMS-1207817, and DMS-1160461}
 \address{Department of Mathematics \\ University of Texas \\ Austin, TX
78712} 
 \email{dafr@math.utexas.edu}
 \author[M. J. Hopkins]{Michael J.~Hopkins}
 \thanks{The work of M.J.H. is supported by the National Science Foundation
under grants DMS-0906194,
DMS-0757293, DMS-1158983}
 \address{Department of Mathematics \\ Harvard University \\ Cambridge, MA 02138}
 \email{mjh@math.harvard.edu}
 \dedicatory{In memory of Dan Quillen}
 \date{March 14, 2013}
 \begin{abstract} 
 We prove that Chern-Weil forms are the only natural differential forms
associated to a connection on a principal $G$-bundle. We use the homotopy
theory of simplicial sheaves on smooth manifolds to formulate the theorem and
set up the proof. Other arguments come from classical invariant theory. We
identify the Weil algebra as the de Rham complex of a specific simplicial
sheaf, and similarly give a new interpretation of the Weil model in
equivariant de Rham theory. There is an appendix proving a general theorem
about set-theoretic transformations of polynomial functors. 
 \end{abstract}
\maketitle

   \section{Introduction}\label{sec:1}
% lastsubsec@000

Invariant theory was studied in the $19^{\textnormal{th}}$~century in the
context of linear representations of algebraic groups.  Given a group~$G$ and
a linear representation~$V$, one seeks polynomials on~$V$ which are invariant
under the action of~$G$.  At around the same time Felix Klein formulated his
\emph{Erlanger Programm}~\cite{K} which, very roughly, defines geometric
concepts as those invariant under a given symmetry group.  For example,
classical Euclidean geometry studies invariants under the Euclidean group of
symmetries of the Euclidean plane~$\EE^2$.  The invariants are no longer
polynomials, but may be a numerical invariant of pairs of points (length), of
triples of points (the isometry class of a triangle), of a polygon (e.g., the
area enclosed), etc.  A broader interpretation of Klein's vision formulates
Riemannian geometry as the study of invariants of Riemannian manifolds under
isometries.  Categorical language enables a precise formulation: there is a
category whose objects are Riemannian manifolds and whose morphisms are
isometries; invariants are functors mapping out of this category, or out of
closely related ones.

The problem we investigate here asks for invariants of principal $G$-bundles
with connection over smooth manifolds, where $G$~is a fixed Lie group.
Specifically, the invariants we seek are differential forms.  Long ago Chern
and Weil showed that conjugation-invariant polynomials on the Lie
algebra~$\mathfrak{g}$ define invariant differential forms.  Our main result
(Theorem~\ref{thm:a2}) is that these are the \emph{only} natural differential
forms one can construct from a $G$-connection.  A similar invariant theory
question was crucial in the initial heat equation approach to the
Atiyah-Singer index theorem as carried out by Gilkey~\cite{G}; see
also~\cite{ABP}.
 
Our focus is not only this specific theorem,\footnote{Some version of this
theorem may already be known, but we could not find a reference.  One novelty
may be Lemma~\ref{thm:a4}; we do not need to assume that the forms we
consider are local functions of the connection---we prove it.} but also the
context we lay out to formulate and prove it.  The ``invariance'' here is
under symmetries of $G$-connections as well as smooth maps of manifolds, so
we need a framework which tracks both.  From another point of view we seek a
universal $G$-connection such that any connection on a principal $G$-bundle
$P\to M$ is pulled back from the universal one.  Universal connections have
been constructed~\cite{NR,Sch} on infinite dimensional
manifolds---see~\cite{Ku,R,DHZ} for further studies---but they have the
drawback that classifying maps are not unique.
In~\S\S\ref{sec:2}--\ref{sec:5} we take the reader on a journey that begins
in an elementary way with these traditional universal objects and leads to
certain ``generalized manifolds'': simplicial sheaves on the category of
smooth manifolds.  We construct a simplicial sheaf\footnote{pronounced ``$B$
nabla $G$''.  The idea of considering $\ENG\to\BNG$ as a universal principal
$G$-bundle with connection surely dates back at least to the early 1970s and
to ideas implicit in~\cite{Br,D}.  We could not find an explicit reference
from that era, however.  The construction does appear in the much more
recent~\cite{FSS}.}~$\BNG$ of $G$-connections and a discrete simplicial sheaf
~$\Omega ^{\bullet }$ of differential forms.  The precise version of our
question becomes a computation: Compute all maps $\BNG\to\Omega ^{\bullet }$.
The actual computation is in~\S\ref{sec:6} and~\S\ref{sec:7}, where we prove
(Theorem~\ref{thm:a2}) that the classical construction of Chern and Weil
captures all differential forms naturally associated to a $G$-connection.
Once the framework is set up, the computation involves only ideas from
differential geometry and invariant theory: no simplicial sheaves.  One piece
of the invariant theory---a proof that set-theoretic transformations of
polynomial functors are polynomial---may be of independent interest; it is
worked out in the appendix.
 
Our work is a new take on Chern-Weil theory and equivariant de Rham theory.
In the world of simplicial sheaves we define the total space~$\ENG$ of the
universal bundle with connection.  We prove (Theorem~\ref{thm:32}) that its
de Rham complex is the \emph{Weil algebra}, the star character in H. Cartan's
treatment~\cite{C1,C2} of Chern-Weil theory.  Traditionally, the Weil algebra
is used as a finite dimensional \emph{model} of the infinite
dimensional de Rham complex of a Hilbert manifold model of~$EG$, or a finite
dimensional approximation thereof.  Here we offer a geometric interpretation
of the Weil algebra as \emph{precisely} the de Rham complex of~$\ENG$.  We
also prove a generalization.  For a $G$-manifold~$X$ we define a version of
the Borel quotient using~$\ENG$ and prove (Theorem~\ref{thm:34}) that the de
Rham complex of this simplicial Borel quotient is precisely the Weil model in
equivariant de Rham theory.

A crucial ingredient in our discussion is abstract homotopy theory, which we
describe in~\S\ref{sec:10}.  Were we to only discuss differential forms, even
in their incarnation as the sheaf of sets~$\Omega ^{\bullet }$, there would
be no need for homotopy theory.  But principal bundles have automorphisms,
usually called gauge transformations, and a fixed $G$-connection may be
stabilized by a nontrivial subgroup of gauge transformations.  So these
objects may appear in several equivalent forms.  We can describe a $G$-bundle
directly as a certain type of fiber bundle~ $P\to X$, or alternatively we can
specify it by an open cover of the manifold~$X$ and transition functions.
There are similar alternative descriptions of a $G$-connection.  Abstract
homotopy theory provides a mechanism for systematically identifying these
alternatives.
 
In recent years abstract homotopy theory has had a profound impact on various
parts of algebraic geometry, as well as on low-dimensional topology.  Here we
use abstract homotopy theory in differential geometry.  There are closely
related contexts in which abstract homotopy theory also plays a crucial role.
For example, generalized differential cohomology groups~\cite{HS} are most
naturally defined in this world; see~\cite{Bu} for a recent exposition.  We
remark that they generalize Cheeger-Simons cohomology groups~\cite{ChS}, which
in turn refine Chern-Weil theory.  Abstract homotopy theory also lies at the
foundation of derived differential geometry~\cite{Sp,Joy}.  In a different
direction, simplicial sheaves provide a good framework in which to define a
general notion of a ``field'' in the sense of classical and quantum field
theory~\cite[Appendix]{FT}.

We offer this paper as a tribute to Dan Quillen.  He introduced abstract
homotopy theory in~\cite{Q1,Q2}, and he also wrote about Chern-Weil theory
in~\cite{Q3,MQ}.  Dan was an exceptionally clear and elegant mathematical
thinker.  He leaves behind a legacy of profound and powerful mathematics
which will continue to inspire for a very long time.

\tableofcontents

   \section{What is a universal connection?}\label{sec:2}
% lastsubsec@000

In this section we motivate the question ``What is a universal connection?'',
for which we begin with some topological analogues.  Let $\Sigma $~be a compact
surface with no boundary.  Its simplest topological invariant is the Euler
number~$\chi (\Sigma )\in \ZZ$, defined for a triangulation of~$\Sigma $ by
Euler's famous formula $\chi (\Sigma )=V-E+F$, where $V$~is the number of
vertices, $E$~the number of edges, and $F$~the number of faces.  For our
purposes we replace this combinatorial definition with one based on a smooth
structure.  We focus on the tangent bundle $\pi \:T\Sigma \to \Sigma $, the
linearization of~$\Sigma $ which assigns to each point $p\in \Sigma $ the
two-dimensional tangent space~$\pi \inv (p)=T_p\Sigma $.  The tangent
bundle~$T\Sigma $ is the union of these two-dimensional real vector spaces,
collected into a smooth manifold.  The Euler number measures the ``twisting''
of the tangent spaces as $p$~varies over~$\Sigma $.  One qualitative
indication that there is twisting for~$\Sigma =S^2$ is the \emph{hairy ball
theorem}, which states that there is no smooth nonzero vector field on the
2-sphere: every hairy sphere has a bald spot!  If there were no twisting,
then we could identify each tangent space with the standard 2-dimensional
real vector space~$\RR^2$, and then promote a nonzero vector at a single
point to a nonzero global vector field.

To obtain a quantitative measurement of the twisting of $T\Sigma \to\Sigma $
we ask: What is the maximally twisted smooth family of two-dimensional real
vector spaces?  One source of twisted smooth families are
\emph{Grassmannians}: any real vector space~ $W$ determines the Grassmannian
~$Gr_2(W)$ of 2-planes in~$W$, named after the $19^{\textnormal{th}}$~century
German mathematician and high school teacher Hermann Grassmann.  The disjoint
union of the two-planes is the total space of a vector bundle $\pi \:E(W)\to
Gr_2(W)$.  It turns out that every two-dimensional real vector bundle is
pulled back from this one as long as we take~$W$ to be \emph{infinite}
dimensional.  Therefore, our quest for a universal real vector bundle of rank
two takes us outside the world of \emph{finite} dimensional manifolds.
Furthermore, even if we allow infinite dimensional manifolds the classifying
manifold is \emph{not} unique.  For example, we can take $W=\RR^{\infty}$
with a direct limit topology and similarly topologize $Gr_2(\RR^{\infty})$.
Or we can take~$W$ to be an infinite dimensional real Hilbert space and
correspondingly construct a Hilbert manifold~$Gr_2(W)$.  There is no
canonical choice for the classifying space.  Rather, there is a theorem in
topology that any two choices are \emph{homotopy equivalent}.  We return to
this issue in~\S\ref{sec:4}.  For now we observe that the cohomology of a
universal parametrizing space~$Gr_2(W)$ is independent of the choice, and for
any choice we locate a universal Euler class ~$\chi \in H^2\bigl(Gr_2(W)
\bigr)$.\footnote{More precisely, the universal Euler class lies in
cohomology twisted by a canonical local system constructed from the
orientations of the two-planes in~$W$.}  Then the universality expresses the
tangent bundle to~$\Sigma $ as a \emph{pullback}
  \begin{equation}\label{eq:1}
  \begin{split}
     \xymatrix{T\Sigma \ar[r]^<(.3){\tilde f}\ar[d] & E(W)\ar[d] \\ \Sigma
     \ar[r]^<(.3)f & Gr_2(W)}
  \end{split}
  \end{equation}
and the Euler number of~$\Sigma $ is the value of~$f^*\chi $ on the
fundamental class of~$\Sigma $---the ``integral'' of~$f^*\chi $ over~$\Sigma
$.  (The pullback property means that $\tilde f$~is an isomorphism
from~$T_p\Sigma $ to the vector space labeled by~$f(p)$, for each~$p\in
\Sigma $.) Not only is $Gr_2(W)$ not unique, but the classifying map~$f$ is
also not unique, though any two are homotopic.

The Euler class of a two-plane bundle is the first example of a
\emph{characteristic class}~\cite{MS}.  More generally, for
any\footnote{Throughout we assume $G$~has finitely many components.} Lie
group~ $G$ we consider a \emph{principal $G$-bundle} $\pi \:P\to M$ over a
smooth manifold~$M$.  By definition $G$~acts freely on the manifold~$P$ with
quotient map~$\pi $, and the action admits local slices.  The principal
bundle associated to $T\Sigma \to\Sigma $ has $G=GL_2\RR$ and the total space
is the set of isomorphisms $\RR^2\to T_p\Sigma $ for all~$p\in \Sigma $.
There are again infinite dimensional universal bundles $EG\to BG$, unique up
to homotopy, and elements of~$H^{\bullet }(BG)$ are universal topological
invariants of principal $G$-bundles.  Characteristic classes are the solution
to a 2-step problem: Find a universal $G$-bundle and compute its cohomology.
 
The problem we consider in this paper is to construct ``differential
geometric characteristic classes''.  To motivate it let's return to our
smooth surface~$\Sigma $ and now suppose it is endowed with a Riemannian
metric~$g$.  The differential geometers of the $18^{\textnormal{th}}$~ and
$19^{\textnormal{th}}$~centuries studied the concrete case of a
surface~$\Sigma \subset \EE^3$ embedded in Euclidean 3-space with the induced
metric.  Befitting the local nature of the metric we now ask not for global
measurements of \emph{topological} twisting, but rather for local
measurements of \emph{geometric} twisting.  Gauss' \emph{Theorema Egregium}
provides a single function $K_g\:\Sigma \to\RR$, the \emph{Gauss curvature},
which is an invariant of the metric and measures the deviation from flatness.
We remark that there is also a canonical measure~$dA_g$ constructed from the
metric---from lengths and angles we compute areas---and the Gauss-Bonnet
theorem asserts that the integral $\frac{1}{2\pi }\int_{\Sigma }K_g\,dA_g$
equals the Euler number~$\chi (\Sigma )$.
 
The curvature is a combination of second derivatives of the metric.  The
Italian school in the late $19^{\textnormal{th}}$~century constructed a new
geometric object constructed from first derivatives, the \emph{Levi-Civita
connection}.  In the $20^{\textnormal{th}}$~century connections were
recognized to have independent interest, and at mid-century Charles
Ehresmann~\cite{E} formulated the notion of a connection~$\Theta $ on a
principal $G$-bundle $P\to M$.  Let $\mathfrak{g}$ denote the Lie algebra
of~$G$.  A \emph{connection} on~$P$ is a 1-form $\Theta \in \Omega
^1(P;\mathfrak{g})$ which satisfies two conditions: (i)~the restriction
of~$\Theta $ to each fiber is the Maurer-Cartan form~$\theta \in \Omega
^1(G;\mathfrak{g})$; and (ii)~if $R_g\:P\to P$ denotes the action of~$g\in
G$, then $R_g^*\Theta =\Ad_{g\inv }\Theta $.  Our problem is to construct
local invariants of connections.  A natural home for these invariants is the
generalization of functions: \emph{differential forms}.  Recall that for any
manifold~$M$ differential forms~$\df M$ have a differential~$d$ which defines
the \emph{de Rham complex}
  \begin{equation}\label{eq:2} 
     \Omega ^0(M) \xrightarrow{\;\; d\;\;} \Omega ^1(M)\xrightarrow{\;\;
     d\;\;}\Omega ^2(M)\xrightarrow{\;\;d\;\;}\cdots 
  \end{equation}

We are led, then, to the following two-step problem: 

  \begin{problem}[]\label{thm:1}
 Construct a universal $G$-connection on a principal $G$-bundle $\ENG\to\BNG$.
  \end{problem}

  \begin{problem}[]\label{thm:2}
 Compute the de Rham complex of ~$\BNG$.
  \end{problem}

\noindent
 The idea of the universal connection is that any connection~$\Theta $ on a
bundle $P\to M$ is pulled back from a map $M\to\BNG$, analogously
to~\eqref{eq:1}.  Once we seek a universal $G$-connection it is natural to
seek a universal de Rham complex as well. 

  \begin{problem}[]\label{thm:3}
 Construct a universal space of differential forms~$\Omega ^{\bullet }$ and
universal de Rham complex. 
  \end{problem}

\noindent
 If we in addition impose \emph{uniqueness} of the classifying map of a
differential form, then we seek some object~$\Omega ^{\bullet }$ such that
for any smooth manifold~$M$ the space of maps $M\to\Omega ^{\bullet }$ is
  \begin{equation}\label{eq:3}
    \Map(M,\Omega ^{\bullet })= \Omega ^{\bullet }(M) .
  \end{equation}
Just as classifying spaces of two-plane bundles and of principal $G$-bundles
are not finite dimensional manifolds, but rather are infinite dimensional
manifolds, so too we construct $\Omega ^{\bullet }$~and $\BNG$~ as
``generalized manifolds''.  We introduce that generalization and solve
Problem~\ref{thm:3} in the next section.  By contrast $\BNG$~is not as rigid
as~$\Omega ^{\bullet }$, hence we defer its construction, which requires
homotopy theory, to~\S\ref{sec:5}.  The solution to Problem~\ref{thm:2} is
stated as Theorem~\ref{thm:a2} and proved in~\S\ref{sec:7}.

   \section{Presheaves and sheaves on manifolds}\label{sec:3}
% lastsubsec@000

Let $\Man$ denote the category whose objects are smooth finite dimensional
manifolds and whose morphisms are smooth maps between manifolds.  The right
hand side of~\eqref{eq:3} is a set attached to each smooth manifold~$M$.
Furthermore, if $f\:M'\to M$ is a smooth map, then there is a pullback map of
differential forms 
  \begin{equation}\label{eq:4}
     \df{M'}\xleftarrow{\;\;\;f^*\;\;}\df M 
  \end{equation}
and the pullback of the composition of two maps is the composition of the
pullbacks.  Let $\Set$~denote the category of sets.  We summarize the
structural properties of differential forms by the statement that
  \begin{equation}\label{eq:5}
     \begin{aligned} \Omega ^{\bullet }\: \Man\op&\longrightarrow \Set \\
      M&\longmapsto \df M\end{aligned} 
  \end{equation}
is a functor.  Here $\op$ denotes the \emph{opposite} category in which all
arrows are reversed: differential forms pull back.

  \begin{definition}[]\label{thm:4}
 A \emph{presheaf on manifolds} is a functor $\Man\op\to\Set$. 
  \end{definition}

\noindent
 In this context we view~$M$ as a ``test manifold'' on which we evaluate the
presheaf.  The presheaf itself is to be considered as a new geometric object
which generalizes a manifold.  To justify that point of view we must first see
that manifolds may be regarded as presheaves.  Let $X$~be a smooth finite
dimensional manifold, and define the associated presheaf~$\sFX$
  \begin{equation}\label{eq:6}
     \begin{aligned} \sFX\:\Man\op&\longrightarrow \Set \\ M &\longmapsto
      \Man(M,X)\end{aligned} 
  \end{equation}
To a test manifold~$M$ this presheaf assigns the set of all smooth maps $M\to
X$, the set of maps from~$M$ to~$X$ in the category~$\Man$.  Throughout we
use standard constructions and notations in categories, for example~
`$C(X,Y)$' for the set of morphisms $X\to Y$ in the category~$C$.

  \begin{remark}[]\label{thm:52}
 The notion of a presheaf is more familiar over a fixed manifold~$X$.  A
presheaf over~$X$ assigns a set to each open set in~$X$ and there are
coherent restriction maps, so it may be viewed as a functor 
  \begin{equation}\label{eq:66}
     \Open(X)\op\longrightarrow \Set
  \end{equation}
on the category whose objects are open subsets of~$X$ and whose morphisms are
inclusions of open sets.  A good general reference on presheaves and sheaves
is~\cite{MM}. 
  \end{remark}

If presheaves on manifolds are meant to generalize manifolds, then we must be
able to do geometry with presheaves, and to begin we define maps between
presheaves, so a category~$\presheaves$ of presheaves.

  \begin{definition}[]\label{thm:5}
 Let $\sF',\sF$ be presheaves on manifolds.  Then a \emph{map} $\varphi
\:\sF'\to\sF$ is a natural transformation of functors.  Thus for each test
manifold~$M$ there is a map $\sF'(M)\xrightarrow{\varphi (M)}\sF(M)$ of sets
such that for every smooth map $M'\xrightarrow{f}M$ of test manifolds the
diagram
  \begin{equation}\label{eq:7}
  \begin{split}
     \xymatrix{\sF'(M')\ar[d]_{\varphi (M')}
     &\sF'(M)\ar[l]_<(.2){\sF'(f)}\ar[d]^{\varphi 
     (M)} \\ \sF(M') &\sF(M)\ar[l]_<(.25){\sF(f)}} 
  \end{split}
  \end{equation}
commutes. 
  \end{definition}

\noindent
 This definition has the nice feature that if the domain presheaf~$\sF'$ is
that of a smooth manifold~$X$, as in~\eqref{eq:6}, then we use~$M=X$ as a
test manifold and so determine $\varphi \:\sFX\to \sF$ by its value
on~$\id_X\in \sFX(X)$, which is an element~$\varphi (\id_X)$ of the
set~$\sF(X)$.  More formally, we have the following.

  \begin{lemma}[Yoneda]\label{thm:6}
 For any presheaf~$\sF$, evaluation on~$X$ determines an isomorphism\newline
$\presheaves(\sFX,\sF)\cong \sF(X)$.
  \end{lemma}

\noindent 
 Here `$\presheaves(\sFX,\sF)$' denotes the set of maps in the category of
presheaves introduced in Definition~\ref{thm:5}.  Because of
Lemma~\ref{thm:6} for any presheaf~$\sF$ we sometimes write an element
of~$\sF(X)$ as a map $X\to \sF$.

  \begin{remark}[]\label{thm:7}
 It is important to observe that smoothness is encoded in the
presheaf~$\sFX$, even though the values of~$\sFX$ are sets with no additional
structure.  For example, a special case of Lemma~\ref{thm:6} is that for any
smooth manifolds~$X,Y$
  \begin{equation}\label{eq:8}
     \presheaves(\sFX,\sF_Y)\cong \sF_Y(X) = \Man(X,Y).
  \end{equation}
In other words, maps $\sFX\to\sF_Y$ of presheaves are precisely \emph{smooth}
maps $X\to Y$ of manifolds.  
  \end{remark}

  \begin{remark}[]\label{thm:71}
 What appears in~\eqref{eq:8} are discrete sets, but the construction
actually remembers much more.  For if $S$~is any smooth manifold, then the
set of smooth maps from $S$~into the function space of maps $X\to Y$ is
$\Man(S\times X,Y)$; see Example~\ref{thm:49} below.
  \end{remark}

  \begin{remark}[]\label{thm:13}
 The map $X\mapsto \sFX$ defines a functor from~$\Man$ into the category of
presheaves on manifolds.  Then \eqref{eq:8} asserts that this functor induces
an isomorphism on Hom-sets, i.e., is ``fully faithful''.  So $\Man$~is a full
subcategory of presheaves, which expresses precisely the sense in which
presheaves are generalized manifolds.
  \end{remark}

Another consequence of Lemma~\ref{thm:6} is that for any smooth manifold~$X$,
we have 
  \begin{equation}\label{eq:9}
     \presheaves(\sF_X,\Omega ^{\bullet })\cong \df X. 
  \end{equation}
Of course, the definition~\eqref{eq:5} is rigged to make this true.  What's
more, in the world of presheaves on manifolds we can define differential
forms on the presheaf~$\Omega ^q$ for each~$q\in \ZZ^{\ge0}$.  For example,
there is a canonical $q$-form
  \begin{equation}\label{eq:10}
     \omq=\id_{\Omega ^q}\:\Omega ^q\longrightarrow \Omega ^q 
  \end{equation}
which to every test manifold~$M$ assigns the identity map on~$\Omega ^q(M)$.
The form~$\omq$ enjoys a tautological \emph{uniqueness} property: if $\omega
\in \Omega ^q(X)$, then there is a \emph{unique} map $\varphi \:\sFX\to\Omega
^q$ such that $\varphi ^*(\omq)=\omega $.  The map~$\varphi $ is defined by
$\varphi (f)=f^*\omega $ for $f\:M\to X$.  It is now straightforward to write
the universal de Rham complex
  \begin{equation}\label{eq:11}
      \quad\Omega ^0 \xrightarrow{\;\;d\;\;} \Omega ^1
     \xrightarrow{\;\;d\;\;} \Omega ^2 \xrightarrow{\;\;d\;\;} \Omega ^3
      \xrightarrow{\;\;d\;\;} \cdots  
  \end{equation}
which on a test manifold~$M$ is the de Rham complex~\eqref{eq:2} on~$M$.  The
complex~\eqref{eq:11} is the solution to Problem~\ref{thm:3}.  Again, we
emphasize that we have constructed a universal object into which there are
\emph{unique} classifying maps.

The reader may feel that we have defined away the problem with no gain.  To
dispel such illusions we retort that any presheaf on manifolds has a de Rham
complex, as illustrated here.

  \begin{theorem}[]\label{thm:8}
 The de Rham complex of~$\Omega ^1$ is isomorphic to 
  \begin{equation}\label{eq:12}
        \RR\xrightarrow{\;\;0\;\;} \RR\xrightarrow{\;\;1\;\;}
     \RR\xrightarrow{\;\;0\;\;} \RR\xrightarrow{\;\;1\;\;} \cdots 
  \end{equation}
In particular, the de Rham cohomology of~$\Omega ^1$ is  
  \begin{equation}\label{eq:13}
     H^{\bullet }_{dR}(\Omega ^1)\cong \begin{cases} \RR,&\bullet
     =0;\\0,&\bullet \not= 0.\end{cases} 
  \end{equation}
  \end{theorem}

\noindent
 The vector space~$\Omega ^q(\Omega ^1)$ has dimension one and is generated
by~$(d\omo)^{\wedge (q/2)}$ if $q$~is even and $\omo\wedge (d\omo)^{\wedge
(q-1)/2}$ if $q$~is odd, where $\omo$~is defined in~\eqref{eq:10}.  This is a
special case\footnote{in which the Lie group is~$G=\TT$, the circle group} of
Theorem~\ref{thm:32}, whose proof appears in~\S\ref{sec:7}.  It does not
appear that this special case has a substantially simpler proof than the
general case.

  \begin{remark}[]\label{thm:9}
 Theorem~\ref{thm:8} unpacks into a concrete statement in ``invariant
theory''.  Namely, $\tau \in \Omega ^q(\Omega ^1)=\presheaves(\Omega ^1,\Omega
^q)$~is a natural construction of a $q$-form from a 1-form.  Thus if $M$~is a
smooth manifold and $\omega \in \Omega ^1(M)$ we have $\tau (\omega )\in
\Omega ^q(M)$, and for any smooth map $f\:M'\to M$ we have $\tau (f^*\omega
)=f^*\tau (\omega )$.  The language of presheaves encodes this naturality
statement, and it opens the way to more intricate definitions and theorems.
  \end{remark}

  \begin{remark}[]\label{thm:11}
 Theorem~\ref{thm:8} shows that the generalized manifold $\Omega ^1$ ~has a
rather simple de Rham complex, e.g., it is finite dimensional in each degree.
By contrast, the space of 1-forms on a positive dimensional ordinary
manifold~$X$ is infinite dimensional.
  \end{remark}

The examples of presheaves we have encountered are determined by local data:
functions and forms are determined by their values on arbitrarily small open
sets.  We abstract that property, which is not satisfied by a general
presheaf.  

  \begin{definition}[]\label{thm:10}
 Let $\sF\:\Man\op\to\Set$ be a presheaf.  Then $\sF$~is a \emph{sheaf} if
for every manifold~$M$ and every open cover~$\{U_\alpha \}$ of~$M$
  \begin{equation}\label{eq:14}
     \sF(M)\longrightarrow
     \prod\limits_{\alpha _0}\sF(U_{\alpha _0})\;
     \substack{\longrightarrow\\\longrightarrow } \;\prod\limits_{\alpha _0
     ,\alpha _1 }\sF(U_{\alpha_0} \cap 
     U_{\alpha_1} )
  \end{equation}
is an equalizer diagram.
  \end{definition}

\noindent
 This is the usual gluing property of a sheaf: given $x_{\alpha_0} \in
\sF(U_{\alpha_0} )$ such that the restrictions of~$x_{\alpha_0} $
and~$x_{\alpha_1} $ to~$U_{\alpha_0} \cap U_{\alpha_1} $ agree, there is a
unique $x\in \sF(M)$ which restricts on ~$U_{\alpha _0}$ to~$x_{\alpha _0}$.
Functions satisfy this gluing property, and more generally $\Omega ^{\bullet
}$ is a sheaf.  We often say $\sF(U)$~is the set of sections of the
sheaf~$\sF$ on the open set~$U$.

Each presheaf $\mathcal F$ has a universal map $\mathcal F\to \ass\mathcal F$
to a sheaf~$\ass\mathcal F$, called the sheafification of~$\mathcal F$.
Universality means that if $\sF'$ is a sheaf and $\sF\to \sF'$ is a map (of
presheaves) then there is a unique sheaf map $\ass \sF\to \sF'$ making the
diagram
  \begin{equation}\label{eq:81}
  \begin{split}
     \xymatrix{ \sF \ar[dr] \ar[r] & \ass \sF \ar[d] \\ & \sF'} 
  \end{split}
  \end{equation}
 commute.  In the language of categories this is expressed by saying that the
forgetful functor $\sheaves \longrightarrow \presheaves $ from the category
of sheaves to the category of presheaves has a \emph{left adjoint} $\ass$.
The functor $\ass$ is the {\em associated sheafification} functor.

Just as presheaves on a fixed topological space have stalks, so too do
presheaves on manifolds.  As all manifolds are locally diffeomorphic to a
ball in affine space, we may as well record only one stalk in each dimension.

  \begin{definition}[]\label{thm:41}
 Let $\sF\:\Man\op\to\Set$ be a presheaf.  For $m\in \ZZ^{\ge0}$ the
\emph{$m$-dimensional stalk} of~$\sF$ is the colimit
  \begin{equation}\label{eq:51}
     \colim_{r\to0} \sF\bigl(B^m(r)\bigr), 
  \end{equation}
where $B^m(r)\subset \AA^m$ is the ball of radius~$r$ about the origin in
$m$-dimensional affine space. 
  \end{definition}

\noindent
 If $r'<r$ there is an inclusion $B^m(r')\hookrightarrow B^m(r)$, and so a
restriction map $\sF\bigl(B^m(r)\bigr)\to \sF\bigl(B^m(r')\bigr)$.  The colimit
is explicitly computed by taking the disjoint union of
all~$\sF\bigl(B^m(r)\bigr)$ and identifying $x\in \sF\bigl(B^m(r)\bigr)$ and
$x'\in \sF\bigl(B^m(r')\bigr)$ if $x$~maps to~$x'$ under restriction.  The
stalks are sets with no additional structure.  For example, the
$m$-dimensional stalk of the sheaf~$\Omega ^{\bullet }$ is the set of germs
of smooth differential forms at the origin in~$\AA^m$.

  \begin{remark}[]\label{thm:42}
 The information contained in a sheaf is essentially the collection of
stalks, one for each nonnegative integer~$m$, and the maps between them
induced from germs of smooth maps between affine spaces.
  \end{remark}

  \begin{remark}[]\label{thm:59}
 The map $\sF\to\ass\sF$ induces an isomorphism on stalks.
  \end{remark}

Function spaces provide another class of examples of sheaves on manifolds.

  \begin{example}[function spaces]\label{thm:49}
 Let $X,Y$ be smooth manifolds.  The space of smooth maps $X\to Y$ may be
given the structure of an infinite dimensional Fr\'echet manifold, but we can 
alternatively work with it as a sheaf~$\sF$.  Namely, for a test manifold~$M$
let $\sF(M)$ be the set of smooth maps $M\times X\to Y$.  There are many
variations.  For example, if we replace~$Y$ by the sheaf~$\Omega ^1$ then
$\sF(M)=\Omega ^1(M\times X)$.  Notice that by promoting 1-forms on~$X$ to a
sheaf, we attach to a test manifold~$M$ 1-forms on the product~$M\times X$,
not partial 1-forms defined only on tangent vectors pointing along~$X$.
  \end{example}

A first attack on Problem~\ref{thm:1} might begin by considering the
presheaf~$\sF$ which to a test manifold~$M$ assigns 
  \begin{equation}\label{eq:15}
     \sF(M) = \{\textnormal{isomorphism classes of $G$-connections
     on~$M$}\} 
  \end{equation}
for a fixed Lie group~$G$.  Let $\mathfrak{g}$~be the Lie algebra of~$G$.  An
element of~$\sF(M)$ is an equivalence class of principal $G$-bundles $P\to M$
with connection~$\Theta \in \Omega ^1(P;\mathfrak{g})$, where two
connections~$\Theta ,\Theta '$ are equivalent if there is a bundle
isomorphism $\varphi \:P'\to P$ covering the identity map on~$M$ such that
$\varphi ^*(\Theta )=\Theta '$.  It is standard to verify that $\sF$~is a
\emph{presheaf} on manifolds.  By contrast, $\sF$~is not a \emph{sheaf}.  For
consider $M=S^1$ with the open cover~$U_1,U_2\subset S^1$ by the complements
of two distinct points.  Then $U_1\cap U_2$~is diffeomorphic to two disjoint
intervals and \eqref{eq:14} reduces to the diagram
  \begin{equation}\label{eq:16}
  \begin{split}
     \xymatrix{\sF(S^1)\ar[r]\ar[d] & \sF(U_1)\ar[d]
     &G/G\ar[r]\ar[d] & \{* \}
     \ar[d]\\ \sF(U_2)\ar[r] &\sF(U_1\cap U_2) &\{*\}\ar[r] &
     \{*,*\}}   
  \end{split}
  \end{equation}
A $G$-connection on~$S^1$ is determined up to isomorphism by the conjugacy
class its holonomy, from which $\sF(\cir)=G/G$ is the set of conjugacy
classes.  (Here $G$~acts on itself by conjugation.)  Any connection on an
interval is isomorphic to the trivial connection.  Since \eqref{eq:16}~fails
to be a pullback diagram, $\sF$~is not a sheaf.

Isomorphism classes of connections do not glue since connections have
automorphisms.  In the preceding example trivial connections on the
intervals~$U_1,U_2$ can glue to a nontrivial connection on~$S^1$, and this
explains the failure of~$\sF$ to be a sheaf.  In the next section we explore
techniques for tracking the automorphisms and so ultimately for gluing
connections.

   \section{Homotopy theory}\label{sec:4}
% lastsubsec@000

In this section we come to grips with the following question. 

  \begin{problem}[]\label{thm:14}
 What mathematical structure~$\sS$ describes the collection of
$G$-connections on a fixed manifold~$M$?
  \end{problem}

\noindent
 Let's first take the simplest case $G=\zt$, the cyclic group of order two.
A principal $\zt$-bundle over~$M$ is simply a double cover $P\to M$, and it
has a unique connection.  So a $\zt$-connection \emph{is} a double cover.
Problem~\ref{thm:14} specializes to: What mathematical structure~$\sS$
encodes all double covers of~$M$?  Notice that the fiber of a double cover
consists of two points, which we can think of as the two points of norm one
in a real one-dimensional vector space equipped with an inner product.  Since
the space of inner products is contractible---any two are related by a
positive scalar---double covers are ``topologically equivalent'' to real line
bundles.  Thus we are returning to the classification problem at the
beginning of~\S\ref{sec:2}, only now for real line bundles in place of real
2-plane bundles.  Consider, then, the Grassmannian~$Gr_1(W)$ of lines in any
real vector space~$W$, also called the \emph{projective space}~$\PP(W)$.  For
$W$~infinite dimensional it is a good model for the collection of all lines,
and we might be tempted to take the space of smooth maps $M\to Gr_1(W)$ as
the answer to Problem~\ref{thm:14}.  But as in~\S\ref{sec:2} there are
different ways to make~$Gr_1(W)$ an infinite dimensional manifold, and for
none of them is there a 1:1~correspondence between maps $M\to Gr_1(W)$ and
real line bundles over~$M$.  So we seek a different approach in which
classifying maps are unique.

The non-uniqueness of classifying maps is due to the fact that a double
cover~$P\to M$ has internal symmetries, namely maps $\varphi \:P\to P$ which
cover the identity map on~$M$.  If $M$~is connected there is a unique
non-identity symmetry, the deck transformation which flips the sheets of the
double cover.  We need a mathematical structure which tracks the symmetries
and, more generally, tracks isomorphisms $\varphi \:P'\to P$ between
different double covers.  One possibility is to organize double covers and
their isomorphisms into a \emph{groupoid}.
 
  \begin{definition}[]\label{thm:38}
 A \emph{groupoid}~$\sG$ is a category in which every arrow is invertible.
Two groupoids $\sG,\sG'$ are \emph{equivalent} if they are equivalent
categories, i.e., if there exist functors $f\:\sG\to \sG'$, $g\:\sG'\to\sG$
and natural equivalences $g\circ f\simeq \id_{\sG}$, $f\circ g\simeq
\id_{\sG'}$.
  \end{definition}

\noindent 
 We write $\sG=\{\sG_0,\sG_1\}$, where $\sG_0$~ is the collection of objects
and $\sG_1$~ the collection of morphisms.  A functor $f\:\sG\to\sG'$ is an
\emph{equivalence} if and only if it is essentially surjective and fully
faithful.  The first condition means that for each~$x'\in \sG'_0$ there
exists~$x\in\sG\mstrut _0$ and $(fx\to x')\in \sG'_1$.  The second means that
for all~$x,y\in \sG\mstrut _0$ the map $f\:\sG(x,y)\to \sG'(fx,fy)$ is a
bijection, where $\sG(x,y)$~is the set of arrows in~$\sG$ from~$x$ to~$y$.

The following criterion will be useful later.  A groupoid~$\sG$ is
\emph{discrete} if for all~$x,y\in \sG_0$ the set~ $\sG(x,y)$ is either empty
or contains a unique element.

  \begin{lemma}[]\label{thm:74}
 Let $\sG,\sG'$ be discrete groupoids.  Then if $f\:\sG\to\sG'$ is surjective
on objects, it is an equivalence of groupoids.
  \end{lemma}

\noindent
 A set may be regarded as a groupoid with only identity arrows, which in
particular is a discrete groupoid.  Conversely, every discrete groupoid is
equivalent to a set.

Returning to double covers, let
  \begin{equation}\label{eq:17}
  \begin{split}
     \begin{aligned} \sG_0=\sG_0(M)&=\textnormal{collection of double covers
     }\pi \:P\to M ;\\
      \sG_1=\sG_1(M)&=\textnormal{collection of commutative
     diagrams}\raise5ex\hbox{\xymatrix{P'\ar[rr]^\varphi_{\cong } 
     \ar[dr]_{\pi'}&&P\ar[dl]^{\pi }\\&M }} 
      \end{aligned} 
  \end{split}
  \end{equation}
Isomorphisms of double covers, which comprise~$\sG_1$, are part of the
structure.  

As an example consider~$M=\cir$.  The groupoid~\eqref{eq:17} is very large;
each double cover of the circle is a distinct element of~$\sG_0(\cir)$.  But
\emph{up to isomorphism} there are only two distinct double covers of the
circle.  One is the trivial double cover $\pi _0\:\cir\times \zt\to\cir$.
Choose a particular nontrivial double cover~$\pi _1$: identify $\cir\subset
\CC$ as the set of complex numbers of unit norm, and let $\pi _1$~be the
squaring map.  Let~$\sG'$ be the groupoid $\sG'_0=\{\pi _0,\pi _1\}$ and
$\sG'_1=\{\id_{\pi _0},\id_{\pi _1},\varphi _0,\varphi _1\}$ has four
elements.  (Here $\varphi _i$~is the deck transformation of~$\pi _i$.)  There
is an inclusion map $g\:\sG'\to \sG$, and we claim it is an equivalence of
groupoids.  To construct a functor $f\:\sG\to\sG'$, for each double cover
$\pi \:P\to\cir$ \emph{choose} an isomorphism $\pi \to\pi _i$ for $i=0$
or~$i=1$ according to whether $\pi $~is trivializable or not.  The
functor~$f$ maps~$\pi $ to~$\pi _i$, and under the chosen isomorphisms any
arrow $\pi \to\pi '$ in~$\sG$ maps to either~$\id_{\pi _i}$ or~$\varphi _i$.
The composition $f\circ g$~is the identity functor on~$\sG'$, and the chosen
isomorphisms are the data of a natural equivalence from~$g\circ f$
to~$\id_{\sG}$.

Returning to Problem~\ref{thm:14}, it is not enough to simply give a
mathematical structure~$\sS$.  We must also discuss a notion of
equivalence~$\simeq$ between two instances of~$\sS$, a relation we term
\emph{weak equivalence}.  This is a key point: \emph{the solution to
Problem~\ref{thm:14} is a pair~$(\sS,\simeq )$.}  Now $\sS=\Grp$ is a
venerable mathematical structure, and with the notion of equivalence in
Definition~\ref{thm:38} it is a valid solution to Problem~\ref{thm:14}.  For
what we do in this paper it is sufficient, and we will use it to good
advantage, but nonetheless we describe a more general solution which applies
more broadly.  As one motivation, most of us are much fonder of the
\emph{geometric} notion $\sS=\Spc$ with weak equivalences defined to be
(weak) \emph{homotopy} equivalences.  There is a direct relationship of
groupoids and spaces.   There is a functor
  \begin{equation}\label{eq:18}
     \Grp\longrightarrow \Spc 
  \end{equation}
which assigns a \emph{classifying space} to each groupoid; see
Definition~\ref{thm:39} below.

  \begin{theorem}[\cite{S}, Proposition~2.1]\label{thm:58} 
 Equivalent groupoids map to homotopy equivalent spaces.
  \end{theorem}

\noindent
 If we apply~\eqref{eq:18} to the groupoid~$\sG(\pt)$ of double covers of a
point, then we obtain a space homotopy equivalent to the projective space
$\RP^{\infty}=\PP(\RR^{\infty})=Gr_1(\RR^{\infty})$ of the infinite
dimensional real vector space~$W=\RR^{\infty}$, with the direct limit
topology.

  \begin{remark}[]\label{thm:16}
 For $M=\pt$ there is a very efficient groupoid~$\sG''$ weakly equivalent
to~$\sG(\pt)$ with $\sG''_0$~a set with a single element and $\sG''_1$~the
cyclic group of order two.  
  \end{remark}

Our solution to Problem~\ref{thm:14} is a mathematical structure~$\SSet$
called \emph{simplicial set}, which sits between groupoid and space: there
are functors
  \begin{equation}\label{eq:19}
  \begin{split}
     \xymatrix{&\SSet \ar[dr]\\ \Grp\ar[ur]\ar[rr]&&\Spc} 
  \end{split}
  \end{equation}
We begin with a formal definition, which also explains the
notation~`$\SSet$'.  Let $\Delta $~be the category whose objects are nonempty
totally ordered finite sets and whose morphisms are order-preserving maps.
It is equivalent to a category with one object for each nonnegative integer.

  \begin{definition}[]\label{thm:17}
 A \emph{simplicial set} is a functor $F\:\Delta \op\to\Set$.  A
map $F'_{\bullet }\to F_{\bullet }$ of simplicial sets is a natural
transformation of functors.  Simplicial sets form a
category~$\SSet$.\footnote{The notation derives from that in topology, where
$B^A$~is the set of maps $A\to B$; lowering the domain in~`$\SSet$' to a
subscript ``dualizes''~$\Delta $ to the opposite category.}
  \end{definition}

\noindent
 If $F_{\bullet }$~is a simplicial set we define the sequence of sets~$F_0,
F_1, F_2, \dots $ by $F_n=F\bigl(\{0,1,2,\dots ,n\}\bigr)$, whence the bullet
subscript in the notation~`$F_{\bullet }$'.  Intuitively, $F_n$~is the
collection of $n$-simplices of the simplicial set~$F_{\bullet }$.  The
order-preserving maps between the canonical totally ordered
sets~$\{0,1,2,\dots ,n\}$ give the diagram
  \begin{equation}\label{eq:20}
  \begin{split}
    \xymatrix@1{F_0\;\ar@{-->}[r]&\;F_1\;\ar@<1ex>[l] \ar@<-1ex>[l]
     \ar@{-->}@<1ex>[r] \ar@{-->}@<-1ex>[r] & \;F_2\;\cdots\ar@<2ex>[l]
     \ar[l] \ar@<-2ex>[l] 
      } 
  \end{split}
  \end{equation}
The left solid arrows are the $n+1$~ \emph{face maps} of an $n$-simplex.  The
right dashed arrows are the $n$~\emph{degeneracy maps} of an $n$-simplex.
The composition laws in~$\Delta $ induce relations among the face and
degeneracy maps.  We proceed directly to some illustrative examples and
recommend~\cite{MP,Fr} for expository accounts and \cite{Cu,Ma,GJ}.

  \begin{example}[groupoids as simplicial sets]\label{thm:18}
 Let $\sG=\{\sG_0,\sG_1\}$ be a groupoid.  The associated simplicial
set~$F(\sG)_{\bullet }$ has $F(\sG)_0=\sG_0$ and $F(\sG)_1=\sG_1$.  In other
words, the 0-simplices of~$F(\sG)_{\bullet }$ are the objects of the groupoid
and the 1-simplices are the arrows.  For~$n>1$ define $F(\sG)_n$~to be the
collection of compositions of $n$~arrows in~$\sG_1$.  The two face maps
$\xymatrix{F(\sG)_0\,&\,F(\sG)_1\ar@<.5ex>[l] \ar@<-.5ex>[l]}$ are the source
and target maps of the groupoid, and the degeneracy map
$\xymatrix{F(\sG)_0\,\ar@{-->}[r]&\,F(\sG)_1}$ assigns the identity
arrow~$\id_\pi $ to each object~$\pi \in \sG_0$.  There is an elegant formal
definition of~$F(\sG)_{\bullet }$.  An object~$S\in \Delta $ is a category
whose objects are the elements of~$S $ and there is a unique morphism $s\to
s'$ if $s\le s'$ in~$S$.  Then the value of $F(\sG)\:\Delta \op\to\Set$
on~$S$ is the collection of functors $S\to\sG$.  Since the $n$-simplices
for~$n>1$ are determined by the 0-~and 1-simplices, the simplicial set
determined by a groupoid only carries information about topology in
dimensions zero and one.
  \end{example}

  \begin{example}[discrete simplicial sets]\label{thm:19}
 If $S$~is any set we promote it to a groupoid~$\sG$ with $\sG_0=\sG_1=S$;
there are only identity arrows.  By Example~\ref{thm:18} this determines a
simplicial set~$S_{\bullet }$ with $S_n=S$ for all~$n$; all simplices of
positive degree are degenerate.  We usually omit the lower bullet in the
notation for a constant simplicial set.  If $T$~is a topological space, we
can form the discrete simplicial set of the underlying point set of~$T$.  A
discrete simplicial set only encodes topology in dimension zero.
\emph{Discrete} simplicial sets are also called \emph{constant} simplicial
sets.
  \end{example}

We can remember more of the topology of a space~$T$ by the \emph{singular
simplicial set}.

  \begin{example}[spaces as simplicial sets]\label{thm:20}
 A space~$T$ determines a simplicial set~$\Sing_{\bullet }T$ defined by
$\Sing_nT=\Spc(\Delta ^n,T)$, the set of continuous maps of the standard
$n$-simplex~$\Delta ^n$ into~$T$.  The face and degeneracy maps are induced
by the corresponding maps of standard simplices.  The simplicial
set~$\Sing_{\bullet }T$ encodes topology in all dimensions.
  \end{example}

  \begin{example}[group actions]\label{thm:23}
 Let $S$~be a set and $G$~a group which acts on~$S$.  There is a
groupoid~$\sG_{\bullet }$ which describes this group action.  Namely, the set
of objects is~$\sG_0=S$ and the set of arrows is~$\sG_1=G\times S$: for
every~$s\in S$ and~$g\in G$ there is an arrow with source~$s\in \sG_0$ and
target~$g\cdot s\in \sG_0$.  The group action defines the composition of
arrows.  The corresponding simplicial set is
  \begin{equation}\label{eq:71}
  \begin{split}
     \xymatrix@1{S\;\ar@{-->}[r]&\;G\times S\;\ar@<1ex>[l]
     \ar@<-1ex>[l] \ar@{-->}@<1ex>[r] \ar@{-->}@<-1ex>[r] & \;G\times G\times
     S\;\cdots\ar@<2ex>[l] \ar[l] \ar@<-2ex>[l] } 
  \end{split}
  \end{equation}

  \end{example}

  \begin{example}[]\label{thm:21}
 Let $X$~be a smooth manifold and~$\sU=\{U_\alpha \}_{\alpha \in A}$ an open
cover.  There is an associated simplicial set~$F(\sU)_{\bullet }$ which
starts off as
  \begin{equation}\label{eq:21}
  \begin{split}
     \xymatrix@1{\coprod\limits_{\alpha _0\in A}U_{\alpha
     _0}\;\ar@{-->}[r]&\;\coprod\limits_{\alpha _0,\alpha _1\in A}U_{\alpha
     _0}\cap U_{\alpha _1}\;\ar@<1ex>[l] \ar@<-1ex>[l] 
     \ar@{-->}@<1ex>[r] \ar@{-->}@<-1ex>[r] & \;\coprod\limits_{\alpha
     _0,\alpha _1,\alpha _2\in A}U_{\alpha 
     _0}\cap U_{\alpha _1}\cap U_{\alpha _2}\;\cdots\ar@<2ex>[l] 
     \ar[l] \ar@<-2ex>[l] }
  \end{split}
  \end{equation}
There a natural map $f\:F(\sU)_{\bullet }\to X_{\bullet }$ to the discrete
simplicial set~$X_{\bullet }$ defined at each level by inclusion.  The
simplicial set~$F(\sU)_{\bullet }$ is derived as in Example~\ref{thm:18} from
a groupoid~$\sG$, where $\sG_i=F(\sU)_i$, $i=0,1$.  The inclusion map~$f$ is
an equivalence of groupoids, as follows immediately from Lemma~\ref{thm:74}.
  \end{example}

  \begin{definition}[]\label{thm:70}
 Let $F_{\bullet }$ be a simplicial set.  Then the set~ $\pi _0(F_{\bullet })$ is
defined as the coequalizer of 
  \begin{equation}\label{eq:113}
  \begin{split}
     \xymatrix@1{F_0\;&\;F_1\;\ar@<.5ex>[l] \ar@<-.5ex>[l]} 
  \end{split}
  \end{equation} 
  \end{definition}

\noindent
 For example, if $F_{\bullet }$~is the simplicial set associated to the
groupoid in Example~\ref{thm:23}, then $\pi _0(F_{\bullet })$~is the set of
orbits of the $G$-action on~$S$.

The functor $\SSet\to\Spc$ in~\eqref{eq:19} is called \emph{geometric
realization}.  We briefly recall the definition~\cite{Mi}.  Define the
standard $n$-simplex
  \begin{equation}\label{eq:49}
     \Delta ^n=\{(t^0,t^1,\dots ,t^n)\in \AA^{n+1}:t^i\ge0,\;t^0+\cdots+
     t^n=1\}. 
  \end{equation}
If $I\in \Delta $~is any nonempty ordered finite set, then it has a unique
isomorphism to~${0,1,\dots ,n}$ for some~$n$, and we define~$\Sigma
(I)=\Delta ^n$.  There is an easy extension to a functor $\Sigma \:\Delta
\to\Spc$.

  \begin{definition}[]\label{thm:39}
 Let $F\:\Delta \op\to\Set$~be a simplicial set.  The geometric
realization~$|F_{\bullet }|$ is the quotient space of the disjoint union
  \begin{equation}\label{eq:50}
     \coprod\limits_{I} \Sigma (I)\times F(I) 
  \end{equation}
by the identifications $(\theta _*t,x)\sim (t,\theta ^*x)$ for all morphisms
$\theta $ in~$\Delta $. 
  \end{definition}

\noindent
 More concretely, one can replace~\eqref{eq:50} by a disjoint union of
$\Delta ^n\times F_n$ and the maps in~\eqref{eq:20} give the gluings.  In
this description $F_n$~parametrizes the set of $n$-simplices in~$|F_{\bullet
}|$; the face and degeneracy maps tell how to glue them together.

  \begin{example}[]\label{thm:44}
 The geometric realization of the discrete simplicial set built from a
set~$S$ (Example~\ref{thm:19}) is canonically isomorphic to~$S$.  The
geometric realization of the simplicial set~$F(\sU)_{\bullet }$ associated to
a cover of a manifold~$X$ is homotopy equivalent to~$X$ with the discrete
topology; this follows from the remarks at the end of Example~\ref{thm:21}
and Theorem~\ref{thm:58}.  More interesting are the geometric realizations in
Example~\ref{thm:20} and Example~\ref{thm:23}: the geometric realization of
the simplicial set~$\Sing_{\bullet }T$ associated to a space~$T$ is homotopy
equivalent to~$T$ (with its given topology), and the geometric realization of
the groupoid built from a discrete group~$G$ acting on a set~$S$ is homotopy
equivalent to the union over the orbits of the classifying spaces of the
stabilizer subgroups.
  \end{example}

  \begin{definition}[]\label{thm:40}
 \ 
  \begin{enumerate}
 \item A continuous map $f\:X\to Y$ of topological spaces is a \emph{weak
homotopy equivalence} if the induced map $f_*\:\pi _0X\to\pi _0Y$ is an
isomorphism of sets and $f_*\:\pi _n(X,x)\to \pi _n(Y,f(x))$ is an
isomorphism for all~$n>0$ and all~$x\in X$.

 \item A map $F_{\bullet }\to F'_{\bullet }$ of simplicial sets is a
\emph{weak equivalence} if the induced map $|F_{\bullet }|\to |F'_{\bullet
}|$ of geometric realizations is a weak homotopy equivalence.
  \end{enumerate}
  \end{definition}

\noindent
 This completes the definition of weak equivalence~$\simeq$ for each of the
three categories in~\eqref{eq:19}, and they are compatible in that the image
of a weak equivalence is a weak equivalence.

The simplicial set which describes $G$-connections on a fixed manifold~$M$ is
associated to the groupoid~$\sG(M)$ defined analogously to~\eqref{eq:17}:
  \begin{equation}\label{eq:22}
  \begin{split}
     \begin{aligned} \sG_0=\sG_0(M)&=\textnormal{collection of
     G-connections~$(\pi ,\Theta )$ where 
     }\pi \:P\to M ,\\ &\qquad \qquad \Theta \in \Omega ^1(P;\mathfrak{g})
     \textnormal{ a connection;}\\ 
     \sG_1=\sG_1(M)&=\textnormal{collection of commutative 
     diagrams}\raise5ex\hbox{\xymatrix{P'\ar[rr]^\varphi_{\cong }
     \ar[dr]_{\pi'}&&P\ar[dl]^{\pi }\\&M}}\\[-4ex]&\qquad \qquad 
  \textnormal{ with }\varphi ^*\Theta =\Theta ' .
     \end{aligned}  \end{split}
 \end{equation}
As in Example~\ref{thm:18} there is an associated simplicial set~$F_{\bullet
}(M)$.  The set of equivalence classes~$\pi _0F_{\bullet }(M)$ of 0-simplices
is the set of equivalence classes of $G$-connections on~$M$.  But, as
desired, $F_{\bullet }(M)$~also tracks symmetries of $G$-connections.

In summary, our response to Problem~\ref{thm:14} is $(\SSet,\simeq)$ with the
definition of weak equivalence~$\simeq$ given in Definition~\ref{thm:40}(ii).

   \section{Simplicial presheaves and weak equivalence}\label{sec:5}
% lastsubsec@  1

We can now address Problem~\ref{thm:1} by mixing homotopy-theoretical
ideas~(\S\ref{sec:4}) with sheaves on manifolds~(\S\ref{sec:3}).  We begin
with the main definitions and then give many examples to illustrate.

  \begin{definition}[]\label{thm:24}
\ 
  \begin{enumerate}
 \item A \emph{simplicial presheaf on manifolds} (or \emph{simplicial
presheaf} for short) is a functor
  \begin{equation}\label{eq:23}
     \sF_{\bullet }\:\Man\to\SSet 
  \end{equation} 

 \item A simplicial presheaf~$\sF_{\bullet }$ is a \emph{simplicial sheaf} if
for each totally ordered finite set~$I\in \Delta $ the presheaf of
sets~$\sF_{\bullet }(I)$ is a sheaf.

 \item For $m\in \ZZ^{\ge0}$ the \emph{$m$-dimensional stalk} of a simplicial
presheaf~$\sF_{\bullet }$ is the simplicial set
  \begin{equation}\label{eq:54}
     \colim_{r\to0} \sF_{\bullet }\bigl(B^m(r)\bigr),
  \end{equation}
where $B^m(r)\subset \AA^m$ is the ball of radius~$r$ about the origin in
$m$-dimensional affine space.  

 \item A map $\sF_{\bullet }\to\sF'_{\bullet }$ of simplicial presheaves is a
\emph{weak equivalence} if the induced map on $m$-dimensional stalks is a
weak equivalence of simplicial sets for each~$m$.

  \end{enumerate}
  \end{definition}

\noindent
 We may restrict to the canonical totally ordered sets $I_n=\{0,1,\dots
,n\}$, as in \eqref{eq:20}, and so view a simplicial presheaf~$\sF_{\bullet
}$ as a sequence~$\sF_n$ of ordinary presheaves.  It is a simplicial sheaf if
each~$\sF_n$ satisfies the sheaf condition in Definition~\ref{thm:10}.  The
$m$-dimensional stalk of~$\sF_{\bullet }$ is the simplicial set whose set of
$n$-simplices is the $m$-dimensional stalk of~$\sF_n$.  Constructions for
simplicial sets carry over to simplicial sheaves.  For example, a sheaf~$\sG$
of groupoids gives rise to a simplicial sheaf by applying the construction in
Example~\ref{thm:18} to~$\sG(M)$ for each test manifold~$M$.

  \begin{example}[discrete simplicial sheaf]\label{thm:45}
 Let $\sF\:\Man\to\Set$ be a sheaf of sets on manifolds.  Then as in
Example~\ref{thm:19} + Example~\ref{thm:18} we can promote~$\sF$ to a
simplicial sheaf~$\stF_{\bullet }$ whose value on a test manifold~$M$ is the
discrete simplicial set with constant value~$\sF(M)$.  We simply denote this
simplicial sheaf as~$\sF$.
  \end{example}

  \begin{example}[representable simplicial sheaves]\label{thm:46}
 Recall from~\eqref{eq:6} that a smooth manifold~$X$ gives rise to a
sheaf~$\sF_X$.  The analogue for simplicial sheaves begins with a
\emph{simplicial manifold}~$\bX$, which is a simplicial set 
  \begin{equation}\label{eq:52}
  \begin{split}
     \xymatrix@1{X_0\;\ar@{-->}[r]&\;X_1\;\ar@<1ex>[l] \ar@<-1ex>[l]
     \ar@{-->}@<1ex>[r] 
     \ar@{-->}@<-1ex>[r] & \;X_2\;\cdots\ar@<2ex>[l]
     \ar[l] \ar@<-2ex>[l] } 
  \end{split}
  \end{equation}
in which each $X_n$~is a smooth manifold and all structure maps are smooth.
Let $\sF_{\bX}$~be the simplicial sheaf whose value on a test manifold is the
simplicial set
  \begin{equation}\label{eq:53}
  \begin{split}
     \xymatrix@1{\Man(M,X_0)\;\ar@{-->}[r]&\;\Man(M,X_1)\;\ar@<1ex>[l]
     \ar@<-1ex>[l] 
     \ar@{-->}@<1ex>[r] \ar@{-->}@<-1ex>[r] & \;\Man(M,X_2)\;\cdots\ar@<2ex>[l]
     \ar[l] \ar@<-2ex>[l] } 
  \end{split}
  \end{equation}
The simplicial sheaf encodes the topology of the smooth manifolds~$X_n$; the
mapping sets in~\eqref{eq:53} are sets of \emph{smooth} maps.  For example,
the 0-dimensional stalk of~$\sF_{\bX}$ is the simplicial set~\eqref{eq:52}
and the $m$-dimensional stalk is the simplicial set of germs of smooth
functions from an $m$-dimensional ball into the simplicial
manifold~\eqref{eq:52}.

As a special case we consider a smooth manifold~$X$ as the constant
simplicial manifold
  \begin{equation}\label{eq:55}
  \begin{split}
     \xymatrix@1{X\;\ar@{-->}[r]&\;X\;\ar@<1ex>[l] \ar@<-1ex>[l]
     \ar@{-->}@<1ex>[r] \ar@{-->}@<-1ex>[r] & \;X\;\cdots\ar@<2ex>[l]
     \ar[l] \ar@<-2ex>[l] } 
  \end{split}
  \end{equation}
where all face and degeneracy maps are identities.  The induced simplicial
sheaf~$\sF_X$ is the representable sheaf~\eqref{eq:6} promoted to a discrete
simplicial sheaf.  We emphasize that each $\sF_X(M)$~is a discrete simplicial
set, so is discrete as a function on totally ordered sets.  But as a function
of~$M$ the simplicial sheaf~$\sF_X$ detects the smooth structure of~$X$. 
  \end{example}

Let $X$~be a smooth manifold and~$\sU=\{U_\alpha \}_{\alpha \in A}$ an open
cover.  The simplicial manifold~\eqref{eq:21} gives rise to a representable
simplicial sheaf~$(\sF_{\sU})_{\bullet }$ as in Example~\ref{thm:46}.  There
is a natural map
  \begin{equation}\label{eq:57}
     \psi \:(\sF_{\sU})_{\bullet }\longrightarrow \sF_X 
  \end{equation}
to the discrete simplicial sheaf~$\sF_X$ defined by inclusions $U_{\alpha
_0}\cap \cdots\cap U_{\alpha _n}\subset X$.  

  \begin{proposition}[]\label{thm:50}
 $\psi \:(\sF_{\sU})_{\bullet }\to\sF_X$ is a weak equivalence of simplicial
sheaves.
  \end{proposition}

  \begin{proof}
 Both the domain and codomain of~$\psi $ are sheaves of groupoids.  The
inclusion map~\eqref{eq:57} on stalks satisfies the hypotheses of
Lemma~\ref{thm:74}, so is an equivalence of groupoids.  (See also
Example~\ref{thm:21}.)  Then Theorem~\ref{thm:58} and Definition~\ref{thm:40}
imply that the associated simplicial sets are weakly equivalent, and so
Definition~\ref{thm:24}(iv) implies that $\psi $~is a weak equivalence of
simplicial sheaves.
  \end{proof}

We come now to our main example, which is the simplicial sheaf that
classifies principal $G$-bundles with connection.

  \begin{example}[$\BNG$]\label{thm:25}
 Fix a Lie group~$G$.  The simplicial presheaf~$\BNG$ of $G$-connections
assigns to each test manifold~$M$ the simplicial set associated the
groupoid~\eqref{eq:22} of $G$-connections on~$M$.  Since connections and
isomorphisms of connections can be glued along open sets, the simplicial
presheaf~$\BNG$ satisfies the sheaf condition.  If $f\:M'\to M$ is a smooth
map of manifolds, then $\BNG(f)$~is the pullback map on $G$-connections and
their isomorphisms.
  \end{example}

  \begin{remark}[]\label{thm:26}
 There is a technical problem with this example and its close cousins below.
Observe that if $M''\xrightarrow{f''}M'\xrightarrow{f'}M$ is a composition of
smooth maps, and $h\:M\to\RR$ a smooth \emph{function}, then the pullback is
strictly associative: $(f'\circ f'')^*h = (f'')^*(f')^*h $.  However, the
pullback of \emph{sets} is not strictly associative.  So if $E\to M$ is a
fiber bundle, then while $(f'\circ f'')^*E$ is \emph{canonically isomorphic}
to~ $ (f'')^*(f')^*E$, these two fiber bundles over~$M''$ are not
\emph{equal}.  This is dealt with using Grothendieck's theory of fibered
categories or alternatively higher categories.
  \end{remark}

A useful tool for verifying weak equivalences is the following.

  \begin{proposition}[]\label{thm:47}
 Let $\psi \:\sF_{\bullet }\to\sF'_{\bullet }$ be a map of simplicial
presheaves such that $\psi (B)\:\sF_{\bullet }(B)\to\sF'_{\bullet }(B)$ is a
weak equivalence for every ball~$B$ in an affine space.  Then $\psi
\:\sF_{\bullet }\to\sF'_{\bullet }$ is a weak equivalence.
  \end{proposition}

\noindent 
 We must show that $\psi $~induces a weak equivalence on stalks.  As a stalk
is the colimit of the values of the simplicial presheaf on balls, and $\psi
$~is a weak equivalence on balls, it suffices to show that this particular
colimit\footnote{The colimit of weak equivalences is a weak equivalence for
arbitrary \emph{filtered} colimits.} of weak equivalences is a weak
equivalence.

Now we come to the total space~$\ENG$ of the universal $G$-bundle with
connection, which is the home of the universal $G$-connection.

  \begin{example}[$\ENG$]\label{thm:27}
 The simplicial sheaf~ $\ENG$ attaches to any test manifold~$M$ the groupoid
of $G$-connections with trivialization, or rather the associated simplicial
set.  So $\ENG(M)_0$ consists of triples~$(\pi ,\Theta ,s)$, where $\pi
\:P\to M$ is a $G$-bundle, $\Theta $~is a connection, and $s\:M\to P$ is a
global section of~$\pi $.  The 1-simplices of~$\ENG(M)$ are isomorphisms
$\varphi \:P'\to P$ of $G$-bundles which preserve the connection and
trivialization.  Higher simplices are compositions of isomorphisms, as in
Example~\ref{thm:18}.  Observe that a principal $G$-bundle with global
trivialization is \emph{rigid}---it has no nontrivial automorphisms.  This
means that $\ENG$~is weakly equivalent to a discrete simplicial
sheaf. Specifically, there are inverse weak equivalences
  \begin{equation}\label{eq:24}
  \begin{split}
     \xymatrix{\ENG \ar@<.5ex>[r]& \Otg\ar@<.5ex>[l]}, 
  \end{split}
  \end{equation}
where $\Otg$~is the discrete simplicial sheaf whose value on a test
manifold~$M$ is $\Omega ^1(M;\mathfrak{g})$.  On~$M$ the top arrow assigns
to~$(\pi ,\Theta ,s)\in \ENG(M)_0$ the 1-form $s^*\Theta \in \Omega
^1(M;\mathfrak{g})$, and the bottom arrow assigns to $\alpha \in \Omega
^1(M;\mathfrak{g})$ the trivial bundle $\pi \:M\times G\to M$ with identity
section~$s$ and connection form~$\Theta =\alpha +\theta $, where $\theta $~is
the Maurer-Cartan form on~$G$.  For each test manifold~$M$ these maps
determine an equivalence of groupoids (Definition~\ref{thm:38})
$\ENG(M)\simeq \Otg(M)$: the composition beginning on the right is the
identity and in the other direction a section~$s$ of a principal bundle $P\to
M$ determines an isomorphism with the trivial bundle.  It follows from
Proposition~\ref{thm:47} that \eqref{eq:24} are weak equivalences of
simplicial sheaves.
  \end{example}

The notion of a smooth Lie group action on a manifold generalizes to
sheaves.  

  \begin{definition}[]\label{thm:73}
 Let $G$~be a Lie group and $\sF$~a sheaf.  A \emph{$G$-action on~$\sF$} is a
smooth map 
  \begin{equation}\label{eq:115}
     a\:\sF_G\times \sF\to\sF 
  \end{equation}
which satisfies the action property:  on any test manifold~$M$, the sheaf map
$a$~defines an action of the group $\sF_G(M)=\Man(M,G)$ on the set~$\sF(M)$.   
  \end{definition}

\noindent
 As in Example~\ref{thm:23} there is an associated action groupoid and so
simplicial sheaf 
  \begin{equation}\label{eq:116}
  \begin{split}
     \xymatrix@1{\sF\;&\;G\times \sF\;\ar@<.5ex>[l]^<(.3){p_1}
     \ar@<-.5ex>[l]_<(.3){p_0} \ar@{-->}@<1ex>[r] \ar@{-->}@<-1ex>[r] &
     \;G\times G\times \sF\;\cdots\ar@<2ex>[l] \ar[l] \ar@<-2ex>[l] \cdots
     } 
  \end{split}
  \end{equation}
where for convenience we write~`$G$' in place of~`$\sF_G$'.  The map~$p_0$ is
projection and $p_1$~is the action map~\eqref{eq:115}.

  \begin{example}[$\BNGt$]\label{thm:28}
 There is a simplicial \emph{pre}sheaf of \emph{trivializable} $G$-bundles
with connection, but since `trivializable' is not a local condition it is not
a simplicial sheaf.  It is the sub-simplicial presheaf of~$\BNG$ whose value
on a test manifold~$M$ consists of pairs~$(\pi ,\Theta )$ such that $\pi
\:P\to M$ admits sections (but no section is specified).  We replace it by a
more explicit simplicial presheaf which \emph{is} a simplicial sheaf.
Observe that if $\pi \:P\to M$ is a principal $G$-bundle with global
trivialization $s\:M\to P$, then any other global trivialization is given by
$s\cdot g\:M\to P$ for a unique $g\:M\to G$.  Now the set of smooth
maps~$M\to G$ is the value of the sheaf~$\sF_G$ of groups on the test
manifold~$M$; see~\eqref{eq:6}.  Under the equivalence~\eqref{eq:24} the
action of $g\:M\to G$ on~$\alpha \in \Omega ^{1}(M;\mathfrak{g})$ is given by
the formula
  \begin{equation}\label{eq:25}
     \alpha \cdot g = g^*\theta + \Ad_{g\inv }\alpha = g\inv dg + g\inv
     \alpha g. 
  \end{equation}
Here $\theta $~is the Maurer-Cartan form on~$G$, and the last expression is
valid only for matrix groups.  So let~$\BNGt$ be the simplicial sheaf
  \begin{equation}\label{eq:26}
  \begin{split}
     \xymatrix@1{\Otg\;\ar@{-->}[r]&\;G\times (\Otg)\;\ar@<1ex>[l] \ar@<-1ex>[l]
     \ar@{-->}@<1ex>[r] \ar@{-->}@<-1ex>[r] & \;G\times G\times
     (\Otg)\;\cdots\ar@<2ex>[l] \ar[l] \ar@<-2ex>[l] } 
  \end{split}
  \end{equation}
where for convenience we write~`$G$' in place of~`$\sF_G$'.  The first two
solid arrows in~\eqref{eq:26} are 
  \begin{equation}\label{eq:31}
  \begin{split}
     \xymatrix@1{\displaystyle\alpha \atop \displaystyle{g^*\theta
     +\Ad_{g\inv }\alpha }\;\;& \;\;g,\alpha 
     \ar@<1ex>@{|->}[l] \ar@<-1ex>@{|->}[l]} 
  \end{split}
  \end{equation}
Note that $\BNGt$~is a sheaf of groupoids, the action groupoid of $G$~acting
on~$\Otg$. 
  \end{example}

The next result will enable us to make explicit computations with~$\BNG$
in~\S\ref{sec:7}.  Define a map 
  \begin{equation}\label{eq:56}
     \psi \:\BNGt\longrightarrow \BNG 
  \end{equation}
as the following map of sheaves of groupoids.  On a test manifold~$M$ it maps
an element $\alpha \in \Omega ^1(M;\mathfrak{g})$ to the trivial bundle
$\pi \:M\times G\to M$ with connection form $\Theta =\alpha +\theta $.  A map
$g\:M\to G$ induces an isomorphism of the trivial bundle with
connection~$\alpha +\theta $ to the trivial bundle with connection~$\alpha
\cdot g+\theta $, where $\alpha \cdot g$ is defined in~\eqref{eq:25}.

  \begin{proposition}[]\label{thm:29}
 $\psi \:\BNGt\to\BNG$~is a weak equivalence of groupoids, hence of
simplicial sheaves.
  \end{proposition}

  \begin{proof}
 We apply Proposition~\ref{thm:47}.  On a ball~$B$ any principal bundle is
trivializable, so $\psi (M)$~is essentially surjective.  Given $\alpha
,\alpha '\in \Omega ^1(M;\mathfrak{g})$ the set of $g\:M\to G$ such that
$\alpha '=\alpha \cdot g$ is in bijection with automorphisms of the trivial
bundle which map $\alpha +\theta $ to $\alpha \cdot g+\theta $, whence $\psi
(M)$~is fully faithful. 
  \end{proof}

We now state our solution to Problem~\ref{thm:1}.  Up to the weak
equivalences~\eqref{eq:24} and~\eqref{eq:56} the group~$G$ (as a simplicial
presheaf) acts freely on~$\ENG$ with quotient~$\BNG$, so $\ENG\to\BNG$ is a
principal $G$-bundle.  Further, there is a canonical $\mathfrak{g}$-valued
1-form on~$\ENG$, that is, a canonical map
  \begin{equation}\label{eq:27}
     \Tu\:\ENG\to \Otg .
  \end{equation}
which is simply the top arrow in~\eqref{eq:24}.  We call~$\Tu$ the
\emph{universal $G$-connection}.  That appellation is justified by the
following result.

  \begin{proposition}[]\label{thm:30}
 Let $\pi \:P\to X$ be a principal $G$-bundle with connection~$\Theta \in
\Omega ^1(P;\mathfrak{g})$.  Then there is a \emph{unique} classifying map
  \begin{equation}\label{eq:28}
  \begin{split}
     \xymatrix{ P\ar[r]^<(.3)f \ar[d]_\pi  & \ENG\ar[d]\\ X\ar[r]^<(.3){\bar f
                }&\BNG} 
  \end{split}
  \end{equation}
such that $f^*\Tu=\Theta $.  
  \end{proposition}

\noindent
 In \eqref{eq:28} we promote~$P$ and~$X$ to discrete simplicial presheaves.
As expected, the construction is completely tautological.

  \begin{proof}
 Let $\tp\:\tP\to P$ denote the pullback of $\pi \:P\to X$ by~$\pi $, and
$\tT=\pi ^*\Theta $ the pullback connection form on~$\tP$.  There is a
canonical section~$\ts\:P\to\tP$ of~$\tp$.  The triple~$(\tp\,,\,\tT
\,,\,\ts) $ is a 0-simplex in~$\ENG(P)$.  On a test manifold~$M$ define
  \begin{equation}\label{eq:29}
  \begin{aligned}
     \Man(M,P)&\xrightarrow{f(M)}&&\ENG(M) \\
    (\phi  \:M\to P)&\,\,\longmapsto &&\phi ^*(\tp\,,\,\tT \,,\,\ts) 
  \end{aligned}
  \end{equation}
Unwinding the definitions we verify $f^*\Tu=\Theta $.  The uniqueness is
clear.  
  \end{proof}

   \section{Abstract homotopy theory}\label{sec:10}
% lastsubsec@000

We are now in a situation best captured by ``abstract homotopy theory'', or
``homotopical algebra'' (\cite{Q1}).  We have a category $\cat C$ (in our
case simplicial presheaves) and a collection $\weq$ of maps in $\cat C$ we
are calling {\em weak equivalences}.  These weak equivalences are not
isomorphisms, but we wish to think of them as being so.  We therefore focus
on the invariants of weak equivalence, or more precisely functors
  \begin{equation}\label{eq:83}
     F:\cat C\longrightarrow  \cat D 
  \end{equation}
with the property that if $X\to Y$ is a weak equivalence in $\cat C$ then
$FX\to FY$ is an isomorphism in~ $\cat D$.  There is a universal such functor
$L:\cat C\to \ho\cat C$ called {\em the localization of $\cat C$ with respect
to~ $\weq$.}  It is characterized uniquely up to unique isomorphism by the
following universal property: for every category $\cat D$, and every functor
$F:\cat C\to \cat D$ taking the maps in $\weq$ to isomorphisms, there is a
unique functor $\ho\cat C\to D$ making the diagram
  \begin{equation}\label{eq:84}
  \begin{split}
     \xymatrix{ \cat C \ar[dr]_{F}\ar[r]^-{L} & \ho\cat C \ar@{-->}[d] \\ &
     \cat D } 
  \end{split}
  \end{equation}
commute.   

  \begin{remark}[]\label{thm:61}
 Another common notation for~$\ho\cat C$ is~$\weq\inv \cat C$.  Our choice
of~`$\ho\cat C$', and the nomenclature `homotopy category' we adopt for it,
emerges in Example~\ref{thm:62} below.
  \end{remark}

The localization $\ho\cat C$ is constructed by freely
adding to $\cat C$, for each $f\in \weq$, a new morphism $f^{-1}$, and
imposing the relations $f^{-1}\circ f=\id$ and $f\circ f^{-1}=\id$.
The issue then becomes to somehow describe the new collection of maps  
  \begin{equation}\label{eq:85}
     \ho\cat C(X,Y)
  \end{equation}
for each~$X,Y\in \cat C$.  In general, there isn't a guarantee that this is
even a set.

Let's look at some examples.   

  \begin{example}[]\label{thm:62}
 Suppose first that $\cat C$ is the
category of CW complexes, and that $\weq$ is the collection of maps
which are weak homotopy equivalences, i.e.  maps $f:X\to Y$ with the
property that for each point $x\in X$ the map of homotopy groups 
  \begin{equation}\label{eq:86}
     \pi_{n}(X,x)\longrightarrow  \pi_{n}(Y,f(x)) 
  \end{equation}
is a bijection for all $n\ge 0$ (Definition~\ref{thm:40}).  Let $\pi\cat C$
be the homotopy category of $\cat C$: the set of maps $\pi\cat C(X,Y)$ is the
quotient of $\cat C(X,Y)$ by the equivalence relation which identifies
homotopic maps.  By the Whitehead Theorem weak equivalences between CW
complexes are homotopy equivalences, so the maps in $\weq$ are sent to
isomorphisms in $\pi\cat C$.  The universal property of $\ho\cat C$ then
provides a unique functor
  \begin{equation}\label{eq:m1}
   \ho\cat C\longrightarrow  \pi\cat C
  \end{equation}
factoring the quotient $\cat C\to \pi\cat C$.   On the other hand, the
inclusions  
  \begin{equation}\label{eq:110}
     \begin{aligned} X\times \{0 \} &\longrightarrow X\times [0,1] \\ X\times
     \{1 \} &\longrightarrow 
      X\times [0,1] \end{aligned} 
  \end{equation}
are in $\weq$, so homotopic maps in $\cat C$ go to the same map in
$\ho\cat C$.  This shows that the functor $\cat C\to
\ho\cat C$ factors uniquely through $\pi\cat C$.  It follows
that~\eqref{eq:m1} is in fact an isomorphism of categories.  Thus the
maps in $\ho\cat C$ in this case may be calculated as homotopy
classes of maps.
  \end{example}

The terminology of homotopical algebra~\cite{Q1} is borrowed from this
example.  In the language of abstract homotopy theory, the class of maps
$\weq$ is called the class of ``weak equivalences'' and the category $\ho\cat
C$ is called the ``homotopy category'' of $\cat C$.  We will now use this
terminology.  Often there is a notion of ``homotopy'' floating around
in~$\cat C$, and we will use the symbol $\pi(X,Y)$ to denote the quotient of
$\cat C(X,Y)$ by the equivalence relation generated by ``homotopy.''
Generally one hopes to describe $\ho\cat C(X,Y)$ in terms of $\pi(X,Y)$.

  \begin{example}[]\label{thm:63}
 Let's look at another example.  Let $R$ be a ring and $\chain R$ the
category of non-negatively graded chain complexes 
  \begin{equation}\label{eq:87}
     \longrightarrow C_{n}\xrightarrow{\;\;d\;\;}{}
     C_{n-1}\longrightarrow \cdots\longrightarrow C_{0}.  
  \end{equation}
We take the class of weak equivalences $\weq$ to be the class of maps
inducing an isomorphism of homology groups.  The notion of
``homotopy'' we have floating around is that of {\em chain homotopy},
and we define $\pi(X,Y)$ to be the set of chain homotopy classes of
maps from $X$ to $Y$.  Let $I$ denote
the chain complex of free abelian groups whose only non-zero terms are
in degrees $0$ and $1$ and in those degrees is given by 
  \begin{equation}\label{eq:111}
     \begin{aligned} \ZZ\{h \} &\longrightarrow  \ZZ\{e_{0},e_{1} \} \\ h
      &\longmapsto  e_{1}-e_{0}.\end{aligned} 
  \end{equation}
Then a chain homotopy is a map $X\otimes I\to Y$.  Now the two
maps 
  \begin{equation}\label{eq:88}
     X\otimes \ZZ\{e_{i} \} \longrightarrow  X\otimes I 
  \end{equation}
are weak equivalences, so chain homotopic maps are identified in $\ho\chain
R$.  This provides a natural map $\pi(X,Y)\to \ho\chain R(X,Y)$.  By basic
homological algebra, if $X$ is a chain complex of \emph{projective}
$R$-modules and $Y\to Z$ is a weak equivalence then $\pi(X,Y)\to \pi(X,Z)$ is
an isomorphism.  Thus in this case the functor
  \begin{equation}\label{eq:89}
     \pi(X,\slot):\chain R\longrightarrow  \abelian 
  \end{equation}
factors through $\ho\chain R$.  By the Yoneda lemma, the identity map
of $X$, regarded as an element of $\pi(X,X)$, gives a natural (in $Y$)
map $\ho\chain R(X,Y)\to \pi(X,Y)$.  It is straightforward to check
that the composites are both the identity.  So when $X$ is a complex
of projectives, then $\ho\chain R(X,Y)$ is given by $\pi(X,Y)$.  For a
general $X$ one can always find a weak equivalence $\tilde X\to X$
from a complex of projectives to $X$ (a projective resolution).  One
then has the sequence of isomorphisms 
  \begin{equation}\label{eq:90}
     \ho\chain R(X,Y) \approx \ho\chain R(\tilde X,Y) \approx \pi(\tilde X,Y). 
  \end{equation}
  \end{example}

Turning to the case of interest to us, let $\spre$ be the category of
simplicial presheaves (or, equivalently, presheaves of simplicial sets) on
the category $\manifolds$ of smooth manifolds.  Recall
Definition~\ref{thm:24}(iv): a map $\sF_{\bullet }\to \sF'_{\bullet }$ of
simplicial presheaves is a {\em weak equivalence} if the induced map of
stalks is a weak equivalence of simplicial sets.
 
We begin our analysis of $\ho\spre(\sF_{\bullet},\sF'_{\bullet})$ with the
special case in which $\sF'_{\bullet}=\sheaf F'$ is a sheaf, regarded as a
constant simplicial presheaf.  Write $\underline{\pi}_{0}\sF_{\bullet}$ for
the sheaf associated to the presheaf
  \begin{equation}\label{eq:91}
     M \longmapsto \pi_{0}\bigl(\sF_{\bullet}(M)\bigr), 
  \end{equation}
so that $\underline{\pi}_{0}\sF_{\bullet}$ is the sheaf associated to the
presheaf obtained as the coequalizer of
  \begin{equation}\label{eq:92}
  \begin{split}
     \xymatrix{ \sF_{1} \ar@<.5ex>[r]^{d_{0}} \ar@<-.5ex>[r]_{d_{1}} & \sF_{0}.
     } 
  \end{split}
  \end{equation}
(See Definition~\ref{thm:70}.)  Two facts about $\underline{\pi}_{0}$ follow
readily from the definition.  One is that a weak equivalence $\widetilde
\sF_{\bullet}\to \sF_{\bullet}$ induces an isomorphism
$\underline{\pi}_{0}\widetilde \sF_{\bullet}\to
\underline{\pi}_{0}\sF_{\bullet}$, so that $\underline{\pi}_{0}$ defines a
functor on $\ho\spre$.  The other is that the set of maps of simplicial
presheaves
  \begin{equation}\label{eq:93}
     \sF_{\bullet}\longrightarrow \sheaf F' 
  \end{equation}
is naturally isomorphic to the set of {\em sheaf} maps  
  \begin{equation}\label{eq:94}
     \underline{\pi}_{0}\sF_{\bullet}\longrightarrow \sheaf F'. 
  \end{equation}
From this it follows that 
  \begin{equation}\label{eq:95}
     \sF_{\bullet}\longmapsto \spre(\sF_{\bullet},\sheaf F') 
  \end{equation}
factors through $\ho\spre$, and as in our analysis of chain complexes, that
the map 
  \begin{equation}\label{eq:114}
     \spre(\sF_{\bullet},\sheaf F')\xrightarrow{\;\;\cong \;\;}
     \ho\spre(\sF_{\bullet},\sheaf F')   
  \end{equation}
is a bijection.

Now suppose that $\sF'_{\bullet}=\sheaf P$ is a presheaf, regarded as a
constant simplicial presheaf, and let $\ass \sheaf P$ be the associated
sheaf.  Then the canonical map $\sheaf P\to \ass P$ is a weak equivalence (it
is an isomorphism on stalks, as stated in Remark~\ref{thm:59}).  It follows
that one may compute $\ho\cat C(\sF_{\bullet}, \sheaf P)$ via the
isomorphisms
  \begin{equation}\label{eq:96}
     \ho\spre(\sF_{\bullet},\sheaf P)\approx \ho\spre(\sF_{\bullet},\ass \sheaf P)\approx
     \spre(\sF_{\bullet},\ass P). 
  \end{equation}
Summarizing, we have

  \begin{proposition}\label{thm:m1}
 Let $\sheaf P$ be a presheaf, regarded as a constant simplicial presheaf,
and $\ass \sheaf P$ be the associated sheaf.  Then for any
$\sF_{\bullet}\in\spre$, the maps
  \begin{equation}\label{eq:97}
     \ho\spre(\sF_{\bullet},\sheaf P)\approx \ho\spre(\sF_{\bullet},\ass
     \sheaf P)\approx \spre(\sF_{\bullet},\ass \sheaf P) 
  \end{equation}
are isomorphisms.  In particular, if $\sheaf P=\sheaf F'$ is a sheaf,
regarded as a constant simplicial presheaf, then the map
  \begin{equation}\label{eq:98}
     \spre(\sF_{\bullet},\sheaf F') \longrightarrow
     \ho\spre(\sF_{\bullet},\sheaf F')  
  \end{equation}
is an isomorphism.  
  \end{proposition}

\noindent
 Note that the domain $\spre(\sF_{\bullet},\sheaf F')$ of~\eqref{eq:98} is the
equalizer of
  \begin{equation}\label{eq:112}
  \begin{split}
     \xymatrix@1{\presheaves(\sF_0,\sF')\;\ar@<.5ex>[r]
     \ar@<-.5ex>[r] &\;\presheaves(\sF_1,\sF').} 
  \end{split}
  \end{equation}

  \begin{remark}[]\label{thm:72}
 If $\sP$ is a constant simplicial presheaf, then $\underline\pi_{0}\sP$ is
just the associated sheaf $\ass \sP$.  When $\sP$ and $\sP'$~ are presheaves,
regarded as constant simplicial presheaves, Proposition~\ref{thm:m1}
therefore provides an isomorphism
  \begin{equation}\label{eq:99}
     \spre(\ass \sP,\ass \sP')\approx \ho\spre(\sP,\sP'). 
  \end{equation}
Put differently, the homotopy theory of simplicial presheaves knows about
sheaves, so even if we were only interested in sheaves and presheaves,
working in the homotopy theory of simplicial presheaves incorporates the
fundamental relationship between them.  This highlights one role played by
abstract homotopy theory.  It can be used to locate objects constrained by
algebraic conditions, like the sheaf condition, within a broader context more
suitable for doing homotopy theory.
  \end{remark}

Proposition~\ref{thm:m1} computes the maps in the homotopy category of
presheaves when the codomain is a constant simplicial presheaf.  This is what
is used in the remainder of this paper.  A general formula is the content of
the Verdier Hypercovering Theorem.  We do not attempt a complete exposition,
and the reader may skip the remainder of this section without penalty.
 
We begin with a motivating example.

  \begin{example}[]\label{thm:64} Let $X$~be a smooth manifold and $\sF_X$
the associated discrete simplicial sheaf.  A map $\sF_X\to \BNGt$ is a
connection on the trivial $G$-bundle over~$X$.  Let ~$\sU$ be an open cover
of~$X$ and $(\sF_{\sU})_{\bullet }$ the associated representable simplicial
sheaf.  Then a map $(\sF_{\sU})_{\bullet }\to\BNGt$ is a $G$-bundle with
connection on~$X$ together with a trivialization on each open set of the
cover.  Homotopic maps give isomorphic $G$-bundles with connection.  Since
any principal $G$-bundle admits local trivializations, the colimit 
  \begin{equation}\label{eq:121}
      \colim_{\sU} \pi \bigl((\sF_{\sU})_{\bullet },\BNGt \bigr) \approx
      \ho\spre\bigl(\sFX, \BNG\bigr) 
  \end{equation}
over all open covers~$\sU$ is in bijection with the set of isomorphism
classes of $G$-connections on~$X$.  Then applying Proposition~\ref{thm:29} we
deduce an isomorphism
  \begin{equation}\label{eq:122}
     \colim_{\sU} \pi \bigl((\sF_{\sU})_{\bullet },\BNGt \bigr) \approx
      \ho\spre(\sFX,\BNGt ) . 
  \end{equation}
This is a special case of the Verdier theorem in which covers, rather than
hypercovers, suffice to compute maps in the homotopy category.
  \end{example}

We now state the general result.  Let $\hypercover_{\sF_{\bullet}}$ be the
category of hypercovers of a simplicial presheaf~ $\sF_{\bullet}$ and
(simplicial) homotopy classes of maps.  Given a simplicial presheaf
$\sF'_{\bullet}$ form
  \begin{equation}\label{eq:104}
     \pi'(\sF_{\bullet},\sF'_{\bullet}) =
     \colim_{\sU_{\bullet }\in\hypercover_{\sF_{\bullet}}} \pi(\sU_{\bullet
     },\sF'_{\bullet}).   
  \end{equation}
Since each hypercovering $\sU_{\bullet}\to \sF_{\bullet}$ is a weak
equivalence, there is a natural map
  \begin{equation}\label{eq:105}
     \pi'(\sF_{\bullet},\sF'_{\bullet}) \longrightarrow
     \ho\spre(\sF_{\bullet},\sF'_{\bullet}).  
  \end{equation}

\begin{theorem}[Verdier hypercovering theorem]  \label{thm:67}
 If $\sF'_{\bullet}$ is stalkwise a Kan complex then the map
  \begin{equation}\label{eq:106}
     \pi '(\sF_{\bullet},\sF'_{\bullet})\longrightarrow
     \ho\spre(\sF_{\bullet},\sF'_{\bullet})  
  \end{equation}
is a bijection. 
\end{theorem}

\noindent
 The Verdier hypercovering theorem is an elaboration of Verdier's
Th\'eor\`eme~7.4.1 in \cite{SGA}.  It was originally formulated in the
above manner in the thesis of Ken Brown~\cite{Br} (see specifically
Example~2, Theorem~1, and Theorem~2).  It was further extended and refined by
Jardine (see~\cite{J} and the references therein).

  \begin{remark}[]\label{thm:68}
  The Verdier hypercovering theorem contains, as a special case, a formula
for describing $\ass\sheaf{P}$ in terms of $\sheaf{P}$.
  \end{remark}

Abstract homotopy theory was introduced by Quillen in~\cite{Q1} and~\cite{Q2}
under the name of ``Model Categories.''  His original applications were to
finding an algebraic model for rational homotopy theory, and for defining the
``cotangent complex'' of a map of commutative rings, which plays the role of
the cotangent bundle when the map is not smooth.  In the original
applications the emphasis was on the systematic comparison of resolutions,
and alongside the weak equivalences Quillen specified two further collections
of maps, the {\em cofibrations} and {\em fibrations}, and some properties
relating the three classes.  In terms of these he defined a notion of
homotopy, and produced the general formula
  \begin{equation}\label{eq:109}
     \ho\cat C(X,Y) \approx  \pi(X', Y'). 
  \end{equation}
Here $X'$ is an object equipped with a weak equivalence $X'\to X$ and having
the property that the unique map it receives from the initial object
$\emptyset\to X$ is a cofibration, $Y'$ is an object equipped with a weak
equivalence $Y\to Y'$ and having the property that its unique map $Y'\to\ast$
to the terminal object is a fibration, and $\pi(X',Y')$ is the quotient of
$\cat C(X',Y')$ by the relation of homotopy.  One thinks of $X'$ as analogous
to a projective resolution of $X$ and $Y'$ as analogous to an injective
resolution of $Y$.  The use of homotopical algebra to implement algebraic
structures originates in the work of
Bousfield~\cite{Bo1,Bo2} and plays an important role
in the study of the moduli spaces of interest in homotopy theory.  The
special role played by the weak equivalences was apparent early on, but it
was Kan et al.~\cite{DHKS} who undertook to do homotopical algebra
solely with the weak equivalences.  A good introduction to Model Categories
can be found in~\cite{DS}.

   \section{The de Rham complex of $\BNG$}\label{sec:6}
% lastsubsec@000

At last the abstractions and tautologies give way to theorems.  

  \begin{definition}[]\label{thm:75}
 Let $\sF_{\bullet }$~be a simplicial presheaf.  The \emph{de Rham complex}
of~$\sF_{\bullet }$ is 
  \begin{equation}\label{eq:117}
     \ho\spre\bigl(\sF_{\bullet }\;,\;\Omega ^0\xrightarrow{\;d\;}\Omega
     ^1\xrightarrow{\;d\;}\cdots \bigr) \cong \ho\spre(\sF_{\bullet },\Omega
     ^0)\xrightarrow{\;d\;} \ho\spre(\sF_{\bullet },\Omega
     ^1)\xrightarrow{\;d\;} \cdots
  \end{equation} 
  \end{definition}

\noindent
 By Proposition~\ref{thm:m1} each term in the complex~\eqref{eq:117} can be
computed as an equalizer~\eqref{eq:112}: 
  \begin{equation}\label{eq:118}
  \begin{split}
     \ho\spre(\sF_{\bullet },\Omega ^n) =\ker\left\{ \xymatrix@1{\Omega
     ^{n }(\sF_0)\;\ar@<.5ex>[r]^>(.7){\rho _0} \ar@<-.5ex>[r]_>(.7){\rho
     _1} &\;\Omega ^{n }(\sF_1)}  \right\} ,
  \end{split}
  \end{equation}
where $\rho _0,\rho _1$ are induced from the structure maps
$\xymatrix@1{\sF_0\;&\;\sF_1\;\ar@<.5ex>[l]^<(.3){p_1}
\ar@<-.5ex>[l]_<(.3){p_0}} $.  

  \begin{remark}[]\label{thm:77}
 It follows from Definition~\ref{thm:75} that weakly equivalent simplicial
presheaves have isomorphic de Rham complexes.  Therefore, the de Rham complex
is a homotopy invariant.  Below we compute the de Rham complex of~$\BNG$
using a convenient weakly equivalent simplicial presheaf, based on
Proposition~\ref{thm:29}.
  \end{remark}

  \begin{proposition}[]\label{thm:54}
 Suppose a Lie group~$G$ acts on a sheaf~$\sF$.  Then the de Rham complex of
the associated simplicial sheaf
  \begin{equation}\label{eq:74}
  \begin{split}
     \xymatrix@1{\sF\;&\;G\times \sF\;\ar@<.5ex>[l]^<(.3){p_1}
     \ar@<-.5ex>[l]_<(.3){p_0} \ar@{-->}@<1ex>[r] \ar@{-->}@<-1ex>[r] &
     \;G\times G\times 
     \sF\;\cdots\ar@<2ex>[l] \ar[l] \ar@<-2ex>[l] \cdots } 
  \end{split}
  \end{equation}
is the equalizer
  \begin{equation}\label{eq:75}
  \begin{split}
     \ker\left\{\xymatrix@1{\Omega ^{\bullet }(\sF)\;\ar@<.5ex>[r]^>(.7){\rho _0}
     \ar@<-.5ex>[r]_>(.7){\rho _1} &\;\Omega ^{\bullet }(G\times \sF)}\right\} .
  \end{split}
  \end{equation}
  \end{proposition}

\noindent 
 This is immediate from \eqref{eq:118}.

The equalizer of~$\rho _0$ and~$\rho _1$ has an interpretation as the
\emph{basic subcomplex} $\Omega ^{\bullet }(\sF)\bas\subset \Omega ^{\bullet
}(\sF)$, which we now define.  We begin by reviewing some standard
constructions in de Rham theory.  Suppose $G$~is a Lie group which acts
smoothly on a manifold~$X$.  Then the infinitesimal action associates to each
$\xi \in \mathfrak{g}=\Lie(G)$ a vector field~$\hx$ on~$X$, and so a
contraction operator
  \begin{equation}\label{eq:79}
     \iota _{\hx}\:\Omega ^{\bullet }(X)\to\Omega ^{\bullet -1}(X) 
  \end{equation}
of   degree~$-1$   and   a   Lie  derivative   
  \begin{equation}\label{eq:80}
     \sL_{\hx}\:\Omega ^{\bullet }(X)\to\Omega ^{\bullet }(X) 
  \end{equation}
of degree~$0$.  They are related by Cartan's formula $\sL_{\hx}=d\iota _{\hx}
+ \iota _{\hx}d$.  Replace~$X$ by  a sheaf~$\sF$.  Fix $\omega  \in \Omega ^{\bullet
}(\sF)$.   Then for any  test manifold~$M$  and $\phi  \in \sF(M)$  we obtain
$\omega (M,\phi )\in \Omega ^{\bullet }(M)$.  The action~$a$ determines a map
  \begin{equation}\label{eq:72}
     G\times M\xrightarrow{\;\id\times \phi \;}G\times \sF\xrightarrow{\quad
     a\quad 
     }\sF\xrightarrow{\quad \omega \quad }\Omega ^{\bullet } 
  \end{equation}
which is the differential form $\eta =a^*\omega (G\times M,\id\times \phi
)\in \Omega ^{\bullet }(G\times M)$.  (As usual we write~`$G$'
for~`$\sF_G$'.)  For~$\xi \in \mathfrak{g}$ define $\iota _\xi \omega \subset
\Omega ^{\bullet -1}(\sF)$ by
  \begin{equation}\label{eq:73}
     \iota _\xi \omega (M,\phi ) = \iota _{\bx}\eta \res{\{e\}\times M}, 
  \end{equation}
where $\bx$~is the vector field on~$G\times M$ defined by~$\xi $ (it points
along the $G$-factor) and $e\in G$~is the identity element.  The Lie
derivative~$\sL_\xi \omega $ is defined by a formula similar
to~\eqref{eq:73}: 
  \begin{equation}\label{eq:76}
     \sL_\xi \omega (M,\phi ) = \sL_{\bx}\eta \res{\{e\}\times M}. 
  \end{equation}
The Lie derivative and contraction satisfy Cartan's formula.

  \begin{definition}[]\label{thm:53}
 Suppose a Lie group~$G$ acts on a sheaf~$\sF$.  Then the differential form
$\omega \in \Omega ^{\bullet }(\sF)$~is \emph{basic} if {\rm(i)}~$a^*\omega
\res{\{g\}\times \sF} =\omega $ for all~$g\in G$, and {\rm(ii)}~$\iota _\xi
\omega =0$ for all~$\xi \in \mathfrak{g}$.
  \end{definition}

\noindent
 Note that~(i) is the condition that $\omega $~be $G$-invariant.  The
differential of a basic form is basic, so basic forms comprise a subcomplex
$\Omega ^{\bullet }(\sF)\bas\subset \Omega ^{\bullet }(\sF)$.  Condition~(ii)
can be rephrased in terms of a global condition on~$G\times \sF$.

  \begin{proposition}[]\label{thm:76}
 The equalizer~\eqref{eq:75} is the basic subcomplex~$\Omega ^{\bullet
}(\sF)\bas$.  
  \end{proposition}

  \begin{proof}
 Recall that $p_0$~is projection and $p_1$~is the action map~$a$; then $\rho
_0,\rho _1$~are the respective pullbacks.  We claim that for any (test)
manifold~$M$ a form~$\omega \in \df M$ is basic if and only if $p_0^*\omega
=p_1^*\omega \in \df{G\times M}$, where $p_0\:G\times M\to M$ is projection
and $p_1\:G\times M\to M$ is the action.  We make two remarks to aid the
reader in the verification.  First, an element~$\xi \in \mathfrak{g}$ induces
a vector field~$\hx$ on~$M$ and a vector field~$\bx$ on~$G\times M$, and
these vector fields are ``related'' by the action map~$p_1$: in other words,
$(p_1)_*(\bx)=\hx$.  It follows that for any~$\omega \in \df M$, we have
$\iota _{\bx}p_1^*\omega =p_1^*\iota _{\hx}\omega $.  Second, a form~$\eta
\in \df{G\times M}$ is pulled back from~$M$ if and only if $\eta $~is
$G$-invariant and $\iota _{\bx}\eta =0$ for all~$\xi \in \mathfrak{g}$.
  \end{proof}

We introduce the following standard differential graded algebra.

  \begin{definition}[]\label{thm:31}
 Let $V$~be a real vector space.  Then the \emph{Koszul complex}~$\Kos V$ built
on~$V$ is the differential graded algebra 
  \begin{equation}\label{eq:58}
     \Kos V={\textstyle\bigwedge} ^{\bullet }V\otimes \Sym^{\bullet }V 
  \end{equation}
 with differential
  \begin{equation}\label{eq:a9}
     d_K(v)= \tv,\quad d_K(\tv)=0,\qquad v \in V=
     {\textstyle\bigwedge} ^1V,\quad \tv\in V=\Sym^1V. 
  \end{equation}  
  \end{definition}

\noindent
 Here `$\tv$'~denotes $v\in V$ regarded as an element of~$\Sym^1V$.  Note
that $\Kos V$ is generated by~$V={\textstyle\bigwedge} ^1V$ as a differential
graded algebra.  We grade the generators by $\deg {\textstyle\bigwedge}
^1V=1$; it follows that $\deg \,\Sym^1V=2$.  It is a standard result that
$\Kos V$ has trivial cohomology:
  \begin{equation}\label{eq:a10}
     H^{\bullet }(\Koss V,d_K) = \RR.
  \end{equation}
This follows for finite dimensional~$V$ (which is all we need) from the
isomorphism $\Kos (V_1\oplus V_2)\cong \Kos V_1\otimes \Kos V_2$ and the case
when $V$~is 1-dimensional.
 
Our first result is the following. 

  \begin{theorem}[]\label{thm:32}
 The de Rham complex of~$\ENG$ is $\bigl(\Kos \mathfrak{g}^*,d_K\bigr)$.   
  \end{theorem}

\noindent
 It follows from~\eqref{eq:a10} that the de Rham cohomology of~$\ENG$ is that
of a contractible manifold.  $\Kos\mathfrak{g}^*$~is called the \emph{Weil
algebra} of the Lie algebra~$\mathfrak{g}$.  Henri Cartan~\cite{C1,C2} used the
Weil algebra as a model for the cohomology of any realization of~$EG$ as a
\emph{space}.  The novelty here is that the Weil algebra is \emph{precisely}
the de Rham complex of the \emph{generalized manifold}~$\ENG$.
 
Next we state the solution to Problem~\ref{thm:2}.

  \begin{theorem}[]\label{thm:a2}
 The de Rham complex of~$\BNG$ is $\bigl((\Sym^{\bullet
}\mathfrak{g}^*)^G,d=0 \bigr)$.
  \end{theorem}

\noindent
 Here $\Sym^{\bullet }(\mathfrak{g}^*)^G$~is the space of Ad-invariant
polynomials on~$\mathfrak{g}$, graded by twice the degree, and the de Rham
differential vanishes on~$\BNG$.  The classical Chern-Weil
homomorphism~\cite{CS} is an injection
  \begin{equation}\label{eq:a8}
     \bigl((\Sym^{\bullet }\mathfrak{g}^*)^G,d=0 \bigr)\longrightarrow
     \bigl(\Omega ^{\bullet }(\BNG),d \bigr). 
  \end{equation}
Namely, given an invariant polynomial of degree~$k$ on~$\mathfrak{g}$ we
apply it to the curvature of a $G$-connection to obtain a closed $2k$-form,
and this construction is local and natural.  Theorem~\ref{thm:a2} asserts
that Chern-Weil forms are the only natural differential forms attached
to a $G$-connection.

The proofs of Theorem~\ref{thm:32} and Theorem~\ref{thm:a2} are given
in~\S\ref{sec:7}.   
 
Let $X$~be a smooth manifold with a left action of the Lie group~$G$.  The
\emph{Borel quotient}, or homotopy quotient, of~$X$ is usually defined as
$X_G=EG\times _GX$, where $EG$~is a contractible space on which $G$~acts freely,
and $X_G$~is the quotient of~$EG\times X$ by the diagonal $G$-action.  We
mimic this construction in the world of simplicial sheaves. 

  \begin{definition}[]\label{thm:33}
 The \emph{simplicial Borel quotient} of~$X$ by~$G$ is the simplicial
sheaf representing the action of~$G$ on $\ENG\times X$.
  \end{definition}

\noindent
 Using the equivalence~\eqref{eq:24} of~$\ENG$ with~$\Otg$, the simplicial
Borel quotient is equivalent to the simplicial sheaf~$\Bqt$ indicated in the
diagram (cf.~\eqref{eq:26})
  \begin{equation}\label{eq:30}
  \begin{split}
     \xymatrix@1{(\Otg)\times X\;\ar@{-->}[r]&\;G\times (\Otg)\times
     X\;\ar@<1ex>[l] 
     \ar@<-1ex>[l] \ar@{-->}@<1ex>[r] \ar@{-->}@<-1ex>[r] & \;G\times G\times
     (\Otg)\times X\;\cdots\ar@<2ex>[l] \ar[l] \ar@<-2ex>[l] } 
  \end{split}
  \end{equation}
The first two solid arrows in~\eqref{eq:30} are 
  \begin{equation}\label{eq:32}
  \begin{split}
     \xymatrix@1{\displaystyle(\alpha,x) \atop \displaystyle{(g^*\theta
     +\Ad_{g\inv }\alpha ,g\inv \cdot x)}\;\;& \;\;(g,\alpha,x)
     \ar@<1ex>@{|->}[l] \ar@<-1ex>@{|->}[l]} 
  \end{split}
  \end{equation}
Note that $\Bqt$~is a sheaf of groupoids.

To state the next theorem we need a preliminary lemma about the Weil algebra.
The coadjoint action~$\Ad^*$ of~$G$ on~$\mathfrak{g}^*$ induces an action
of~$G$ on the Weil algebra~$\Kos\mathfrak{g}^*$.   

  \begin{lemma}[]\label{thm:55}
 For each $\xi \in \mathfrak{g}$ the action~\eqref{eq:26} of~$G$ on $\Otg$
induces a contraction~$\iota _\xi $ on $\Omega ^{\bullet }(\Otg)\cong
\Koss^{\bullet }\mathfrak{g}^*$ which satisfies
  \begin{equation}\label{eq:44}
     \begin{aligned} \iota _\xi \lambda &= \langle \xi ,\lambda
      \rangle,\qquad &\lambda &\in \mathfrak{g}^*={\textstyle\bigwedge}
      ^1\mathfrak{g}^* \\ \iota _\xi \tla&=-\Ad_\xi ^*\lambda ,\qquad
     &\tla&\in \mathfrak{g}^*=\Sym^1\mathfrak{g}^*\end{aligned} 
  \end{equation}
on generators. 
  \end{lemma}

\noindent
 We defer the proof to~\S\ref{sec:7}. 

  \begin{remark}[]\label{thm:56}
 The contraction~\eqref{eq:44} is usually a definition (\cite{C1,C2},
\cite[\S5]{MQ}, \cite[\S3.2]{GS}).  In our approach with sheaves it is a
computation from the general definition~\eqref{eq:73}. 
  \end{remark}

  \begin{theorem}[]\label{thm:34}\ 
   \begin{enumerate}
 \item For any smooth manifold~$X$ the de Rham complex of $X\times (\Otg)$ is
$\Omega (X;\Koss\mathfrak{g}^*)^{\bullet }$ with differential the sum of the
de Rham differential~$d_X$ on ~$X$ and the Koszul differential~$d_K$
in~\eqref{eq:a9}.

 \item The de Rham complex of the simplicial Borel quotient~$\Bqt$
in \eqref{eq:30} is the basic subcomplex of $\Omega
(X;\Koss\mathfrak{g}^*)^{\bullet }$ with differential $d_X+d_K$.
  \end{enumerate}
  \end{theorem}

\noindent
 Part~(i) would follow immediately from Theorem~\ref{thm:32} if we prove that
the de Rham complex of a Cartesian product of sheaves is the (completed)
tensor product of the de Rham complexes of the factors.  We give a direct
proof in~\S\ref{sec:7}.  Part~(ii) is an immediate corollary of part~(i) and
Proposition~\ref{thm:76}.  The complex in Theorem~\ref{thm:34}(ii) is the
\emph{Weil model} for equivariant cohomology; see~\cite{C1,C2},
\cite[\S4]{AB}, \cite[\S5]{MQ}, \cite[\S4]{GS}.  We realize it precisely as
the de Rham complex of a generalized manifold, the simplicial sheaf~$\Bqt$.

   \section{Proofs}\label{sec:7}
% lastsubsec@  1

We prove the following slight generalization of Theorem~\ref{thm:32}.
(Recall the weak equivalence~\eqref{eq:24}.)

  \begin{theorem}[]\label{thm:a3}
 Let $V$~be a finite dimensional real vector space.  Then the de Rham complex
of $\OtV$ is $(\Kos V^*,d_K)$.  
  \end{theorem}

\noindent 
 More precisely, we define a homomorphism of differential graded algebras 
  \begin{equation}\label{eq:a22}
     \eta \:\Kos V^* \longrightarrow  \Omega ^{\bullet }(\OtV) 
  \end{equation}
as follows.  To $\ell \in V^*={\textstyle\bigwedge} ^1V^*$ we assign the
1-form $\eta (\ell )\: \OtV\to\Omega ^1$ characterized by
  \begin{align}
     \eta (\ell )\:\sum\limits_{i}\alpha_i \otimes
      v_i&\longmapsto \sum\limits_{i}\langle v_i,\ell \rangle\,\alpha
     _i\qquad \alpha _i\in \Omega ^1,\quad v_i\in V.\label{eq:33}\\ 
      \intertext{This determines the entire homomorphism~$\eta $.  For
example, it follows that } 
      \eta (\ell _1\wedge \ell _2)=\eta
      (\ell _1)\wedge \eta (\ell _2)\:\sum\limits_{i}\alpha_i \otimes
      v_i&\longmapsto \sum\limits_{i,j}\langle v_i,\ell _1\rangle\langle
      v_j,\ell _2\rangle\,\alpha _i\wedge \alpha _j\label{eq:34}\\ \eta
     (\tl)=d\eta (\ell 
      ) \:\sum\limits_{i}\alpha_i \otimes v_i&\longmapsto
      \sum\limits_{i}\langle v_i,\ell \rangle \,d\alpha _i\label{eq:35}
  \end{align}
Recall from~\eqref{eq:a9} the notation $\tl\in V^*=\Sym^1V^*$.  Notice that
the map $\Omega ^1\otimes V\to\Omega ^{\bullet }$ attached to an element
of~$\Kos V^*$ is polynomial but not usually linear.  Also, these formulas are
most easily understood by pulling back to a test manifold~$M$ via a map
$M\to\OtV$, in which case $\alpha _i\in \Omega ^1(M)$.  Theorem~\ref{thm:a3}
asserts that $\eta $~is an isomorphism.

We present the proof of Theorem~\ref{thm:a3} as a series of lemmas.  Recall
that if $E\to M$ is a vector bundle, then for each integer~$q\ge 0$ there is
an associated bundle of $q$-jets $J^qE\to M$ defined as follows.  Let
$I_p\subset C^{\infty}(M)$ denote the ideal of functions which vanish at~$p$.
Then the fiber of~$J^qE$ at~$p$ is $\Omega ^0(M;E)/I_p^{q+1}\Omega ^0(M;E)$,
the quotient of the space of sections of~$E$ by the space of sections which
vanish at~$p$ to order~$\ge q+1$.  An element $\omega \in \Omega ^q(\OtV)$
is, for each test manifold~$M$, a map of sets $\omega _M\:\Omega ^1(M;V)\to
\Omega ^q(M)$, functorial in~$M$.

  \begin{lemma}[]\label{thm:a4}
 Fix $\omega \in \Omega ^q(\OtV)$.  For $\alpha \in \Theta ^1(M;V)$ the value
of~$\omega _M(\alpha )$ at a point~$p\in M$ depends only on the $q$-jet
of~$\alpha$ at~$p$.
  \end{lemma}

  \begin{proof}
 Assume first~$q=0$ and $\alpha \in \Omega ^1(M;V)$.  Let
$i_p\:\{p\}\hookrightarrow M$ denote the inclusion.  Then the presheaf
property (naturality) implies $\omega _M(\alpha )\res p = i_p^*\omega
_M(\alpha )= \omega _{\{p\}}(i_p^*\alpha ) = \omega _{\pt}(0) =\omega
_M(0)\res p$.  Thus $\omega _M(\alpha )$~is a constant function on~$M$
independent of~$\alpha $ (so doesn't even depend on the 0-jet of~$\alpha $).

Now let~$q>0$.  Naturality implies that if $i_U\:U\hookrightarrow M$ is the
inclusion of a neighborhood of~$p$ and $\alpha _1,\alpha _2\in \Omega
^1(M;V)$ satisfy $i_U^*\alpha _1=i_U^*\alpha _2$, then $i_U^*\omega _M(\alpha
_1)=\omega _U(i_U^*\alpha _1) = \omega _U(i_U^*\alpha _2) = i_U^*\omega
_M(\alpha _2)$.  In other words, $i_U^*\omega _M(\alpha )$ depends only
on~$i_U^*\alpha $.  Next, we claim that~$\omega _M(0)=0$.  For we have just
proved that if~$p\in M$, the value of~$\omega _M(0)\res p$ can be computed by
restricting to a neighborhood of~$p$, which we may as well assume is a
neighborhood of the origin in affine $m$-space~$\AA^m$.  Then since the
1-form~$\alpha =0$ is invariant under the linear group~$GL_m\RR$ acting
on~$\AA^m$ fixing the origin, so too is the $q$-form~$\omega _M(0)\res
{p=0\in \AA^m}\in {\textstyle\bigwedge} ^q(\RR^m)^*$. Since~$q>0$ the only
$GL_m\RR$-invariant $q$-form is zero.
 
Suppose that $\alpha _1,\alpha _2\in \Omega ^{1}(M;V)$ have identical
$q$-jets at~$p\in M$, so for some functions~$f_0,f_1,\dots ,f_q$ which vanish
at~$p$ and some $\beta \in \Omega ^1(M;V)$ we have 
  \begin{equation}\label{eq:a12}
     \alpha _2 = \alpha _1 + f_0f_1\cdots f_q\beta .
  \end{equation}
Let $\ta_i,\tb$ be the pullbacks of~$\alpha _i,\beta $ to~$M\times
\AA^{q+1}$, and let $t_0,t_1,\dots ,t_q$ be the standard coordinates
on~$\AA^{q+1}$.  Then $\alpha _i,f_0f_1\cdots f_q\beta $ are the pullbacks of
$\ta_i,t_0t_1\cdots t_q\tb$ via the map $id_M\times (f_0,f_1,\dots
,f_q)\:M\to M\times \AA^{q+1}$.  Observe that this map sends~$p$ to~$(p,0)$.
By naturality it suffices to show that
  \begin{equation}\label{eq:a13}
     \omega _{M\times \AA^{q+1}}(\ta_2)\res{(p,0)}\quad \textnormal{and}\quad
     \omega _{M\times \AA^{q+1}}(\ta_1 + t_0t_1\cdots t_q\tb) \res{(p,0)}
  \end{equation}
agree as $q$-forms on~$T_pM\times \RR^{q+1}$.  Decomposing with respect to
the standard basis of~$\RR^{q+1}$ and using multilinearity we
see\footnote{Proposition~\ref{thm:m3} is a general abstraction of this
assertion.} that it suffices to evaluate the $q$-forms~\eqref{eq:a13} on
$q$-vectors which have vanishing component along the
$i^{\textnormal{th}}$~axis in~$\RR^{q+1}$ for some~$i$.  By naturality those
evaluations can be made by restricting the forms~$\omega (\ta_2)$ and~$\omega
(\ta_1 + t_0t_1\cdots t_q\tb)$ to the submanifold~$N_i\subset M\times
\AA^{q+1}$ defined by~$t_i=0$.  But these $q$-forms agree on~$N_i$, since the
1-forms~$\ta_2$ and~$\ta_1 + t_0t_1\cdots t_q\tb$ restrict to equal 1-forms
on~$N_i$.
  \end{proof}

Lemma~\ref{thm:a4} implies that $\omega \in \Omega ^q(\Omega ^1\otimes V)$ is
determined by its value at the origin of vector spaces~$W\in \Man$, and
furthermore for~$\alpha \in \Omega ^1(W;V)$ the value of~$\omega \mstrut
_W(\alpha )$ at the origin of~$W$ is computed by a map (of sets)
  \begin{equation}\label{eq:a14}
     \tw\mstrut _W\:J^q(W;W^*\otimes V)\longrightarrow {\textstyle\bigwedge}
     ^qW^*,  
  \end{equation}
where $J^q(W;W^*\otimes V)$~is the finite dimensional vector space of
$q$-jets of elements of~$\Omega ^1(W;V)$ at the origin.  Furthermore,
\eqref{eq:a14} is functorial in~$W$, in particular for linear maps.  Let
$\Vect$ be the category of finite dimensional real vector spaces and linear
maps.  Then
  \begin{equation}\label{eq:123}
     \begin{aligned} W&\longmapsto J^q(W;W^*\otimes V) \\ W&\longmapsto
      {\textstyle\bigwedge} ^qW^*\end{aligned} 
  \end{equation}
are polynomial functors $\Vect\op\to\Vect$, and the first is reduced.  (See
Appendix~\ref{sec:11} for definitions.)  The map~\eqref{eq:a14} is a
set-theoretic transformation between these two functors.  The following
statement is then a direct consequence\footnote{As Theorem~\ref{thm:mm1}
applies to covariant functors, and \eqref{eq:123}~are contravariant,
precompose with the contravariant functor $W\mapsto W^*$ before applying the
theorem.} of Theorem~\ref{thm:mm1}.

  \begin{lemma}[]\label{thm:a5}
 The map~$\tw_{W}$ is a polynomial of degree~$\le q$. 
  \end{lemma}

Next, use the isomorphism 
  \begin{equation}\label{eq:a18}
     J^q(W;W^*\otimes V)\,\cong\, \bigoplus \limits_{j=0}^{q}
     \,\bigl(\Sym^jW^*\otimes 
     W^*\otimes V \bigr)
  \end{equation}
to write the polynomial~$\tw_W$ \eqref{eq:a14} as a \emph{linear} map
  \begin{equation}\label{eq:a19}
     \TW\:\bigoplus \limits_{i=0}^{q} \,\Sym^i\left( \bigoplus \limits_{j=0}^{q}
     \,\bigl(\Sym^jW^*\otimes W^*\otimes V \bigr) \right)\longrightarrow
     {\textstyle\bigwedge} ^qW^*. 
  \end{equation}

  \begin{lemma}[Weyl]\label{thm:35}
 A nonzero $GL(W)$-invariant linear map $\otimes ^NW^*\to
{\textstyle\bigwedge} ^qW^*$ has~$N=q$ and is a multiple of the
antisymmetrization map. 
  \end{lemma}

  \begin{proof}
 That $N=q$ follows from invariance under the scaling
subgroup~$\RR^\times\subset GL(W)$.  A theorem of Weyl asserts that $\otimes
^qW^*$~ is a direct sum of irreducible representations of~$GL(W)$ of
multiplicity one.  Antisymmetrization is an equivariant map $\otimes
^qW^*\to{\textstyle\bigwedge} ^qW^*$, so any other map must be a multiple of
it. 
  \end{proof}

  \begin{lemma}[]\label{thm:a6}
 $\TW$~factors through $\bigoplus\limits_{i=0}^q\Sym ^{i}\bigl(W^*\otimes
V\;\oplus\;{\textstyle\bigwedge} ^2W^*\otimes V\bigr)$.
  \end{lemma}

  \begin{proof}
 Naturality implies that \eqref{eq:a19}~is invariant under~$GL(W)$.  The
domain of~$\TW$ is a direct sum of quotients of subspaces~$U$ of vector
spaces of the form~$(\otimes ^kW^*)\otimes (\otimes ^\ell V)$ for
some~$k,\ell $.  By Weyl's theorem any equivariant map must have~$k=q$ and be
a multiple of antisymmetrization $\otimes ^qW^*\to{\textstyle\bigwedge}
^qW^*$ tensored with a multilinear form on~$V$.  The restriction
of~\eqref{eq:a19} to terms with~$j\ge 2$ must vanish since for these terms
the lift of~\eqref{eq:a19} to a map $U\subset (\otimes ^kW^*)\otimes (\otimes
^\ell V)\to {\textstyle\bigwedge} ^qW^*$ has~$k\ge2$ and would be symmetric
in at least two of the factors of~$W^*$ in the domain, whereas
antisymmetrization is completely antisymmetric.  Similarly, for~$j=1$ it must
factor through the quotient~${\textstyle\bigwedge} ^2W^*$ of~$W^*\otimes
W^*$.
  \end{proof}

\noindent
 It follows from Lemma~\ref{thm:a6} that for each finite dimensional real
vector space~$W$ there is a linear map 
  \begin{equation}\label{eq:59}
     L\mstrut _W\:\df\OtV\longrightarrow \Hom\mstrut _{GL(W)}\Bigl( \Sym
     \bigl(W^*\otimes V\bigr)\;\otimes \; 
     \Sym\bigl({\textstyle\bigwedge} ^2W^*\otimes 
     V\bigr)\;;\; {\textstyle\bigwedge} ^{\bullet }W^*\Bigr),
  \end{equation}
where `$\Sym$'~denotes the entire symmetric algebra.
As $W$~varies both of
  \begin{equation}\label{eq:67}
     \begin{aligned} F_1\: W&\longmapsto \Sym \bigl(W^*\otimes
      V\bigr)\;\otimes \; \Sym\bigl({\textstyle\bigwedge} ^2W^*\otimes
      V\bigr) \\ F_2\: W&\longmapsto {\textstyle\bigwedge} ^{\bullet
      }W^*\end{aligned} 
  \end{equation}
are polynomial functors $\Vect\op\to\Vect$, and \eqref{eq:59}~defines a
linear map~$L$ from~$\Omega ^{\bullet }(\OtV)$ to the vector space of natural
transformations $F_1\to F_2$.
  Recall the linear map~$\eta $ in~\eqref{eq:a22}. 

  \begin{lemma}[]\label{thm:51}
 $L\circ \eta $~is injective. 
  \end{lemma}

  \begin{proof}
 The Koszul algebra~$\Kos V^*$ is doubly graded with homogeneous components 
  \begin{equation}\label{eq:60}
     \Koss^{p,q}V^* = {\textstyle\bigwedge} ^pV^*\otimes \Sym^qV^*,\qquad
     p,q\in \ZZ^{\ge0}. 
  \end{equation}
Identify $\Koss^{p,q}V^*$~with the space of multilinear maps $\phi
\:V^{\times (p+q)}\to\RR$ which are skew-symmetric in the first $p$~arguments
and symmetric in the last $q$~arguments.  The codomain~$A$ of~$L$ is doubly
graded with homogeneous components
  \begin{equation}\label{eq:61}
     A^{p,q}(W) = \Hom\mstrut _{GL(W)}\Bigl( \Sym ^{p }\bigl(W^*\otimes
     V\bigr)\;\otimes \;\Sym^{q }\bigl({\textstyle\bigwedge} ^2W^*\otimes
     V\bigr)\;;\; {\textstyle\bigwedge} ^{p+2q }W^*\Bigr),\qquad p,q\in
     \ZZ^{\ge0}. 
  \end{equation}
(Invariance under scaling dictates the degree~$p+2q$.)  The map $L\circ \eta
$ preserves the bigrading.  Fix a bidegree~$p,q$.  The image of a multilinear
map 
$\phi \in\Koss^{p,q}V^*$ under $L_W\circ \eta $ is the multilinear map
  \begin{multline}\label{eq:62}
     [w^*_1\otimes v_1],\dots ,[w^*_p\otimes v_p],[(w^*_{p+1}\wedge
     w^*_{p+2}),v_{p+1}],\dots , [(w^*_{p+2q-1}\wedge w^*_{p+2q}),v_{p+q}]
     \\ \longmapsto \phi (v_1,\dots ,v_p,v_{p+1},\dots
     ,v_{p+q})\;w^*_1\wedge \cdots\wedge w^*_{p+2q}, 
  \end{multline}
where $v_i\in V$, $w^*_i\in W^*$.  To prove that $L\circ \eta $~is injective
in bidegree~$p,q$ choose $\dim W=p+2q$, fix a basis $w_1^*,\dots ,w_{p+2q}^*$
of~$W^*$, and evaluate~\eqref{eq:62} on all~$v_i\in V$ to see that if
$(L_W\circ \eta )\phi =0$, then~$\phi =0$.
  \end{proof}

  \begin{proof}[Proof of Theorem~\ref{thm:a3}]
 Lemma~\ref{thm:a4} and the subsequent arguments prove that $L$~is injective.
To prove Theorem~\ref{thm:a3}, which asserts that $\eta $~is an isomorphism,
it suffices now to show that $L\circ \eta $~is surjective.  We again fix a
bidegree~$p,q$ and a vector space~$W$.  Let $h\in A^{p,q}(W)$
(see~\eqref{eq:61}).  Then $h$~lifts to an element~$\tih$ of
     \begin{multline}\label{eq:64} 
       \Hom\mstrut _{GL(W)}\Bigl( \Tens ^{p
       }\bigl(W^*\otimes V\bigr)\;\otimes \;\Tens^{q }\bigl(\otimes
       ^2W^*\otimes V\bigr)\;;\; {\textstyle\bigwedge} ^{p+2q }W^*\Bigr) \\
       \cong \Hom\mstrut _{GL(W)}\Bigl( \Tens ^{p+2q }(W^*)
       \;\otimes \;\Tens^{p+q }(V)\;;\; {\textstyle\bigwedge} ^{p+2q
     }W^*\Bigr) \end{multline}  
By Lemma~\ref{thm:35}, $\tih$~ must have the form 
  \begin{multline}\label{eq:65}
     [w^*_1\otimes v_1],\dots ,[w^*_p\otimes v_p],[(w^*_{p+1}\wedge
     w^*_{p+2}),v_{p+1}],\dots , [(w^*_{p+2q-1}\wedge w^*_{p+2q}),v_{p+q}] \\
     \longmapsto \phi(v_{1},\dots
     ,v_{p+q})\;w^*_1\wedge \cdots\wedge w^*_{p+2q}, 
  \end{multline}
for some $\phi \in \Tens^{p+q}(V^*)$.  The fact that $\tih$~ factors
through~$A^{p,q}(W)$ in~\eqref{eq:61} implies that $\phi $~is skew-symmetric
in the first $p$~variables and symmetric in the last~$q$, so $h$~ is in the
image of $\Koss^{p,q}V^*\to A^{p,q}(W)$.
  \end{proof}

The following is a slight generalization of Theorem~\ref{thm:34}(i).

  \begin{corollary}[]\label{thm:a7}
 Let $V$~be a finite dimensional real vector space and $Y$~a smooth
manifold.  Then the de Rham complex of~$Y\,\times \,(\Omega ^1\otimes V)$ is
$\Omega (Y;\Koss V^*)^{\bullet }$ with differential the sum of the de Rham
differential on~$Y$ and the Koszul differential~\eqref{eq:a9}. 
  \end{corollary}

\noindent
 The de Rham complex of a Cartesian product of manifolds is a completed
tensor product of their individual de Rham complexes, and this corollary
would follow from that statement for generalized manifolds (where there is no
completion since the de Rham complex of~$\OtV$ is finite dimensional).

  \begin{proof}
 We indicate the modifications to the proof of Theorem~\ref{thm:a3} to
accommodate the factor of~$Y$.  Define a homomorphism
  \begin{equation}\label{eq:68}
     \eta \mstrut _Y\:\Omega (Y;\Koss V^*)^{\bullet }\longrightarrow \Omega
     ^{\bullet 
     }(Y\,\times\, \OtV) 
  \end{equation}
as the composition 
  \begin{equation}\label{eq:69}
     \Omega (Y;\Koss V^*)^{\bullet }\xrightarrow{\;\eta \;} \Omega
     \bigl(Y;\Omega (\OtV) \bigr)\xrightarrow{\;\wedge \;}\Omega (Y\,\times\,
     \OtV). 
  \end{equation}
We must prove that $\eta \mstrut _Y$~is an isomorphism.  In Lemma~\ref{thm:a4} for
each test manifold~$M$ we have a pair~$\phi ,\alpha $ consisting of a map
$\phi \:M\to Y$ and a 1-form $\alpha \in \Omega ^1(M;V)$; now $\omega
_M=\omega _M(\alpha ,\phi )$.  The locality argument implies that to
compute~$i^*_U\omega _M(\alpha ,\phi )$ for~$U\subset M$ a small open set
containing~$p\in M$, we can replace~$Y$ by an open set in a vector
space~$V'$.  In the paragraph containing~\eqref{eq:a12} define $\tf\:M\times
\AA^{q+1}\to V'$ as the composition $\phi \circ \pi _1$.  If $\phi _1,\phi
_2$~have the same $q$-jet at~$p$, then $\tf_2=\tf_1+t_0t_1\cdots
t_q\tilde\psi $ for some $\psi \:M\to V'$.  A similar argument to that
surrounding~\eqref{eq:a13} shows that $\omega _M(\alpha ,\phi )\res p$ only
depends on the $q$-jet of~$\phi $ at~$p$, and so we take~$V'=T_{y}Y$, where
$y=\phi (p)$.  Then in~\eqref{eq:a14} we replace~$J^q(W;W^*\otimes V)$ with
the $q$-jets of elements of~$\Omega ^1(W;V)\times \Map(W;T_yY)$ at the origin
of~$W$.  Lemma~\ref{thm:a5} is unchanged.  The isomorphism \eqref{eq:a18}~is
replaced by
  \begin{equation}\label{eq:a23}
     J^q(W;W^*\otimes V\;\oplus \;T_yY)\,\cong\, \bigoplus \limits_{j=0}^{q}
     \,\bigl(\Sym^jW^*\otimes 
     (W^*\otimes V \;\oplus \; T_yY \bigr) 
  \end{equation}
In the statement of Lemma~\ref{thm:a6} there is an extra term~$W^*\otimes T_yY$
under~$\Sym^i$ for each~$i$, and \eqref{eq:59} is now
  \begin{equation}\label{eq:a24}
     L_{y,W}\:\Omega ^{\bullet }(Y\,\times \,\OtV)\longrightarrow
     \Hom_{GL(W)}\Bigl( \Sym(W^*\otimes 
     V\;\oplus 
      \;{\textstyle\bigwedge} ^2W^*\otimes V\;\oplus \;W^*\otimes
     T_yY)\,;\,{\textstyle\bigwedge} ^{\bullet }W^* \Bigr)
  \end{equation}
The composite~$L_{y,W}\circ \eta \mstrut _Y$ preserves a triple grading: 
  \begin{multline}\label{eq:70}
      \Omega ^r(Y;\Koss^{p,q}V^*) \longrightarrow \\
      \Hom_{GL(W)}\Bigl( \Sym^p(W^*\otimes V)\;\otimes \Sym^q(
      \;{\textstyle\bigwedge} ^2W^*\otimes V)\;\otimes \;\Sym^r(W^*\otimes
      T_yY)\,;\,{\textstyle\bigwedge} ^{p+2q+r }W^* \Bigr) 
  \end{multline} 
The proofs that $L_y\circ \eta \mstrut _Y$~is injective and surjective are
similar to those above, so are omitted.
  \end{proof}

  \begin{proof}[Proof of Lemma~\ref{thm:55}]
 We use~\eqref{eq:73}.  So if $\alpha \:M\to\Otg$ is a $\mathfrak{g}$-valued
1-form on a test manifold~$M$, and $\lambda \in \mathfrak{g}^*$, then
by~\eqref{eq:33} the pullback of $\eta (\alpha )$ to~$M$ is the scalar 1-form
$\langle \alpha ,\lambda \rangle\in \Omega ^1(M)$.  The action~\eqref{eq:31}
of~$G$ gives the 1-form $\langle \theta +\Ad_{g\inv }\alpha ,\lambda
\rangle\in \Omega ^1(G\times M)$, and the contraction with~$\xi \in G$
along~$\{e\}\times M$ is~$\langle \xi ,\lambda \rangle$, as claimed in the
first formula of~\eqref{eq:44}.  By~\eqref{eq:35} the functional $\lambda \in
\mathfrak{g}^*$ also determines, via~$\eta $, a scalar 2-form $\langle
d\alpha ,\lambda \rangle\in \Omega ^2(M)$, and the action of~$G$ produces the
2-form
  \begin{equation}\label{eq:78}
     \bigl\langle \frac 12[\theta ,\theta ] +\Ad_{d(g\inv )}\wedge  \alpha
     +\Ad_{g\inv 
     }d\alpha \;,\;\lambda \bigr\rangle\in \Omega ^2(G\times M). 
  \end{equation}
Only the second term contributes to the restriction to $\{e\}\times M$ of the
contraction with~$\xi $, which equals $\langle -\Ad_\xi \alpha ,\lambda
\rangle =\langle \alpha ,-\Ad_\xi^* \lambda \rangle\in \Omega ^1(M)$, the
1-form corresponding to $-\Ad_\xi^* \lambda \in \mathfrak{g}^*$.  This proves
the second line of~\eqref{eq:44}.
  \end{proof}

  \begin{proof}[Proof of Theorem~\ref{thm:a2}]
 By Proposition~\ref{thm:29} and the fact that Definition~\ref{thm:75} takes
place in the homotopy category, de Rham complex of~$\BNG$ equals that
of~$\BNGt$.  Now $\BNGt$~is the simplicial sheaf~\eqref{eq:26} which
represents the action of~$G$ on~$\Otg$.  By Proposition~\ref{thm:54} and
Proposition~\ref{thm:76} its de Rham complex is the basic subcomplex of
$\Omega ^{\bullet }(\Otg)$.  By Theorem~\ref{thm:32} the latter is the Weil
algebra $(\Koss^{\bullet }\mathfrak{g}^*,d_K)$.  Thus we are reduced to
computing the basic subcomplex of the Weil algebra, which is standard.

Following~\cite[\S5]{MQ}, choose a basis~$\{e_i\}$ of~$\mathfrak{g}$ and
corresponding dual basis~$\{e^i\}$ of~$\mathfrak{g}^*$.  Let $\iota _i,\theta
^i$ be interior multiplication by~$e_i$ and exterior multiplication by~$e^i$,
respectively, on~${\textstyle\bigwedge} ^{\bullet }\mathfrak{g}^*$.  Then it
is easy to see that $\prod_{i}(1-\theta ^i\iota _i)$ is projection
onto the \emph{horizontal} elements of~${\textstyle\bigwedge} ^{\bullet
}\mathfrak{g}^*$---those which satisfy Definition~\ref{thm:53}(ii)---and that
the image of this projector is in fact ${\textstyle\bigwedge}
^0\mathfrak{g}^*\subset {\textstyle\bigwedge} ^{\bullet }\mathfrak{g}^*$.
Now ${d\theta ^i}\subset \Sym^1\mathfrak{g}^*\subset \Koss^{\bullet
}\mathfrak{g}^*$, and we set
  \begin{equation}\label{eq:77}
     \Omega ^i=d\theta ^i+\frac 12 f^i_{jk}\theta ^j\wedge \theta ^k, 
  \end{equation}
where $[e_j,e_k]=f^i_{jk}e_i$ and we use the summation convention.  A short
computation from~\eqref{eq:44} shows that $\iota _\ell \Omega ^i=0$ for
all~$i,\ell $.  By a change of basis we may identify the Weil algebra as the
exterior algebra on the span of~$\{\theta ^i\}$ tensor the symmetric algebra
on the span of~$\{\Omega ^i\}$.  It follows that the horizontal elements of
the Weil algebra form the subspace ${\textstyle\bigwedge}
^0\mathfrak{g}^*\otimes \Sym^{\bullet }\mathfrak{g}^*$, where the second
factor is the symmetric algebra on the span of~$\{\Omega ^i\}$.  Therefore,
the basic subalgebra of the Weil algebra are the $G$-invariants in that
symmetric algebra, as claimed.
  \end{proof}

  \begin{remark}[]\label{thm:57}
 The computation of the basic subcomplex of the Weil algebra is a special
case of the computation of the basic subcomplex in Theorem~\ref{thm:34}(ii),
which results in the \emph{Cartan model} for equivariant de Rham cohomology;
see \cite[\S5]{MQ}, \cite[\S4]{GS}. 
  \end{remark}

\appendix

   \section{Transformations of polynomial functors}\label{sec:11}
% lastsubsec@000

%\subsection{Polynomial functors}

Polynomial functors are intimately related to Schur's representation theory
of the symmetric group; for example, see~\cite[\S I,Appendix]{Mac}.  In this
appendix we prove that every set-theoretic natural transformation of
polynomial functors is polynomial.

Let $\vect$ be the category of finite dimensional vector spaces
over\footnote{With the exception of the proof of Lemma~\ref{thm:m2}, the
arguments in this appendix work over any field of characteristic
zero.}~$\RR$.

\begin{definition}
A functor $F:\vect\to \vect$ is {\em polynomial of degree $d$} if for every $V$ and every $f_{1},\dots, f_{n}\in
\End{V}$, the map  
  \begin{equation}\label{eq:124}
     F(\lambda_{1}f_{1}+\dots + \lambda_{n}f_{n}) 
  \end{equation}
is a polynomial of degree $d$ in~$\lambda _1,\dots ,\lambda _n$ with
coefficients in $\End(F(V))$.  A functor $F$ is {\em homogeneous of degree
$d$} if the above polynomial is homogeneous of degree $d$.
\end{definition}

\noindent
 To keep the language simple we make the convention that a polynomial of
degree $d$ might also be a polynomial of lower degree.

Suppose that $F$ is polynomial functor of degree $d$, and write  
  \begin{equation}\label{eq:125}
     F(\lambda\id_V) = \sum_{i=0}^{d} \lambda^{i}\mathbf{e}_{i}(V), \qquad
     \lambda \in k,\quad \mathbf{e}_i(V)\in \End\bigl(F(V) \bigr).  
  \end{equation}
Using 
  \begin{equation}\label{eq:126}
      F(\lambda_{1}\lambda_{2}\id) = F((\lambda_{1} \id)\circ(\lambda_{2}\id))
     = F(\lambda_{1}\id)\circ F(\lambda_{2}(\id)) 
  \end{equation}
one easily checks that the $\mathbf{e}_{i}(V):F(V)\to F(V)$ are orthogonal
idempotents.  Write $F_{i} =\mathbf{e}_{i}F$.  Then $F_{i}$ is homogeneous of
degree $i$, and $F=\bigoplus F_{i}$.

Here is a useful fact about polynomial functors.   For a subset
$I\subset \{1,\dots,n \}$ let  
  \begin{equation}\label{eq:127}
      V^{I} =\{(v_{1},\dots, v_{n})\in V^{n}: v_{i} = 0, i \ne I \} 
  \end{equation}
be the ``$I$-axis,''  and $\epsilon\mstrut _{I}:V^{n}\to V^{n}$ the projection
operator to $V^{I}$.   We write $\epsilon_{i}$ instead of
$\epsilon_{\{i \}}$, so that  
  \begin{equation}\label{eq:128}
      \epsilon_{i}(v_{1},\dots, v_{n}) = (0,\dots, v_{i}, \dots, 0), 
  \end{equation} 
  \begin{equation}\label{eq:129}
      \epsilon\mstrut _{I} = \sum_{i\in I} \epsilon_{i} 
  \end{equation}
and 
\begin{equation}
\label{eq:m2}
\id \mstrut _{V^n}= \epsilon_{1}+\dots + \epsilon_{n}.   
\end{equation}
Write $|I|$ for the number of elements of $I$.

\begin{proposition}
\label{thm:m3} Suppose that $F:\vect\to\vect$ is a polynomial functor
of degree $d$, and let $n>d$.  For every non-zero $x\in F(V^{n})$ there is
a subset $I\subset \{1,\dots,n \}$ with $|I|\le d$ and $F(\epsilon\mstrut
_{I})(x)\ne 0$. 
Equivalently, the product of the restriction maps 
  \begin{equation}\label{eq:130}
     \prod_{|I|=d}F(\epsilon \mstrut _{I})\: F(V^{n}) \to \prod_{|I|=d}
     F(V^{I})  
  \end{equation}
is a monomorphism.
\end{proposition}

 \begin{proof} 
 Since $F$ is polynomial we have \begin{equation} \label{eq:mm1}
F(\lambda_{1}\epsilon_{1}+\dots+\lambda_{n}\epsilon_{n}) = \sum_{J}
\lambda^{J} e\mstrut _{J}
\end{equation}
in which $J=(j_{1}, \dots, j_{n})$, $\lambda^{J}=
\lambda_{1}^{j_{1}}\dots \lambda_{n}^{j_{n}}$, and $e\mstrut _{J}$ is an
endomorphism of $F(V^{n})$.  As above, the $e\mstrut _{J}$ are orthogonal
idempotents.  Setting all of the $\lambda_{i}=1$ and
using~\eqref{eq:m2} one sees that  
  \begin{equation}\label{eq:131}
      \sum_{J}e\mstrut _{J} = \id. 
  \end{equation}
Let $x$ be a non-zero element of $F(V^{n})$.  Apply~ \eqref{eq:131} to~$x$ to
conclude that $e\mstrut _{J}(x)\ne 0$ for some $J$.  Since the degree of $F$
is $d$, at most $d$ of the $j_{i}\in J$ are non-zero.  Let $I=\{i\mid
j_{i}\ne 0 \}$.  Set
  \begin{equation}\label{eq:153}
     \lambda _i=\begin{cases} 1,&i\in I;\\0,&i\notin I\end{cases} 
  \end{equation}
in~\eqref{eq:mm1} to deduce
  \begin{equation}\label{eq:132}
      e\mstrut _{J}\bigl(F(\epsilon\mstrut _{I} )(x)\bigr) = e\mstrut
      _{J}\bigl(e\mstrut _{J}(x) + \cdots \bigr) = e\mstrut _{J}(x)
      \not= 0, 
  \end{equation}
whence $F(\epsilon\mstrut _{I} )(x)\not= 0$.
\end{proof}

  \begin{definition} \label{def:m1}
 Suppose that $V$ and $W$ are real vector spaces.  A function $f:V\to W$ is
{\em polynomial (of degree $d$)} if for every set ${v_{1},\dots, v_{n}}$ of
elements of $V$, the map
  \begin{equation}\label{eq:133}
      f(\lambda_{1}v_{1}+\cdots +\lambda_{n}v_{n}) 
  \end{equation}
is a polynomial in the $\lambda_{i}$ (of degree $d$) with coefficients in
$W$, i.e., there exists a polynomial $g(t_{1},\dots, t_{n})\in
W\otimes\mstrut _{\RR}{\RR[t_{1},\dots, t_{n}]}$ with the property that for
all $\lambda_{1},\dots, \lambda_{n}\in \RR$, and all $v_{1},\dots, v_{n}\in
V$
  \begin{equation}\label{eq:134}
      f(\lambda_{1}v_{1}+\cdots +\lambda_{n}v_{n}) = g(\lambda_{1}, \dots,
     \lambda_{k}). 
  \end{equation} 
 \end{definition}

The following lemma gives a useful criterion for a map to be polynomial.

\begin{lemma} \label{thm:m2} Let $k>0$ be an integer.  Let $f:V\to W$ be a
map, and assume that for every finite set ${v_{1},\dots, v_{n}}$ of elements
of $V$ the map
  \begin{equation}\label{eq:135}
      f(\lambda_{1}^{k}v_{1}+\dots +\lambda_{n}^{k}v_{n}) 
  \end{equation}
is a polynomial in the $\lambda_{i}$ (of degree $k d$) with coefficients
in $W$.  Then $f$~is a polynomial map of degree~$d$.
\end{lemma}

  \begin{proof}
 By choosing bases of~$V$ and~$W$ we immediately reduce to the case
$f\:\RR^m\to\RR$ for $m=\dim V$.  The hypothesis implies that 
  \begin{equation}\label{eq:149}
     f(\lambda _1^k,\dots ,\lambda _m^k) = \sum\limits_{r_1,\dots
     ,r_m}a_{r_1\cdots r_m}\,\lambda _1^{r_1}\cdots \lambda _m^{r_m} 
  \end{equation}
is a polynomial of degree~$kd$ in~$\lambda _1,\dots ,\lambda _m\in \RR$ and 
  \begin{equation}\label{eq:150}
     f(\mu _1^k-\nu _1^k,\dots ,\mu _m^k-\nu _m^k) =
     \sum\limits_{\substack{p_1,\dots ,p_m\\q_1,\dots ,q_m}}b_{p_1\cdots
     p_mq_1\cdots q_m}\,\mu _1^{p_1}\cdots \mu 
     _m^{p_m}\nu _1^{q_1}\cdots \nu _m^{q_m} 
  \end{equation}
is a polynomial of degree~$kd$ in~$\mu _1,\dots ,\mu _m,\nu _1,\dots ,\nu
_m\in \RR$.  For~$|\mu _i|\ge|\nu _i|$ set $\lambda _i=(\mu _i^k-\nu
_i^k)^{1/k}$.  Comparing~\eqref{eq:149} and~\eqref{eq:150} we find for
each~$r_1,\dots ,r_m$ that
  \begin{equation}\label{eq:151}
     a_{r_1\cdots r_m}(\mu _1^k-\nu _1^k)^{r_1/k}\cdots(\mu _m^k-\nu
     _m^k)^{r_m/k} = \sum\limits_{p_i+q_i=r_i}b_{p_1\cdots p_mq_1\cdots
     q_m}\mu _1^{p_1}\cdots \mu _m^{p_m}\nu _1^{q_1}\cdots \nu _m^{q_m} .
  \end{equation}
If $a_{r_1\cdots r_m}\not= 0$ we conclude that 
  \begin{equation}\label{eq:152}
     (\mu _1^k-\nu _1^k)^{r_1/k}\cdots(\mu _m^k-\nu _m^k)^{r_m/k} 
  \end{equation}
is a homogeneous polynomial in $\mu _1,\dots ,\mu _m,\nu _1,\dots ,\nu _m$ on
the region where~$|\mu _i|\ge|\nu _i|$ for all~$i$.  Apply Taylor's theorem
at any point of that region to deduce that each of $r_1,\dots ,r_m$ is
divisible by~$k$ (else \eqref{eq:152}~is not a polynomial).  It follows
from~\eqref{eq:151} that $b_{p_1\cdots p_mq_1\cdots q_m}$~vanishes unless
each of $p_1,\dots ,p_m,q_1,\dots ,q_m$ is divisible by~$k$.  If $k$~is odd,
the vanishing now implies from~\eqref{eq:149} that $f(x_1,\dots ,x_m)$~is a
polynomial of degree~$d$ in~$x_1,\dots ,x_m$, but if $k$~is even we only
deduce this on the region where all~$x_i\ge0$.  From~\eqref{eq:150} we see
that $f(y_1-z_1,\dots ,y_m-z_m)$ is (i)~a polynomial of degree~$d$
in~$y_1,\dots ,y_m,z_1,\dots ,z_m$ and (ii)~a polynomial of degree~$d$ in
$y_1-z_1,\dots ,y_m-z_m$ in the region where all~$y_i\ge z_i$.  From these
two facts it follows that (ii)~holds for all~$y_1,\dots ,y_m,z_1,\dots z_m$,
whence $f(x_1,\dots ,x_m)$~is a polynomial of degree~$d$ in~$x_1,\dots
,x_m$. 
  \end{proof}

  \begin{definition}[]\label{thm:78}
  A polynomial functor
  \begin{equation}\label{eq:136}
      F=\bigoplus_{i\ge 0} F_{i} 
  \end{equation}
is {\em reduced} if $F_{0}=0$.
  \end{definition}

We now prove the main result of this appendix.

\begin{theorem}
\label{thm:mm1}
 Let $F$ and $G$ be polynomial functors over~$\RR$, and suppose that $F$ is
reduced.   Any set-theoretic transformation $T:F\to G$ is polynomial.  If
$G$~is polynomial of degree~$d$, then $T$~is polynomial of degree~$d$.
\end{theorem}

\noindent
 A \emph{set-theoretic} transformation of functors $T\:F\to G$ is a natural
transformation of the underlying set-valued functors.  In other words, the
map $T(V)\:F(V)\to G(V)$, $V\in \Vect$, is a map of sets (which is not
assumed to be linear).

\begin{proof}
Let $V$ be a vector space, $v_{1},\dots, v_{n}\in F(V)$.  We wish to show
that  
  \begin{equation}\label{eq:137}
      T_{V}(\lambda_{1} v_{1} + \dots + \lambda_{n}v_{n}) \in G(V) 
  \end{equation}
is a polynomial function of $\lambda_{i}\in k$.  For each $i$, write
  \begin{equation}\label{eq:138}
      v_{i} = \sum v_{i}^{(j)}, 
  \end{equation}
with $v_{i}^{(j)}\in F_{j}(V)$.   Then since 
  \begin{equation}\label{eq:139}
      \sum \lambda_{i} v_{i} = \sum \lambda\mstrut _{i}v_{i}^{(j)} 
  \end{equation}
we might as well assume from the outset that each $v_{i}$ is homogeneous in
the sense that $v_{i}\in F_{k_{i}}(V)$ for some $k_{i}$.  Since $F$ is
reduced, none of the $k_{i}$ is zero.  We may therefore choose an integer $k$
which is divisible by all the $k_{i}$.  We will show that
  \begin{equation}\label{eq:140}
      T_{V}(\lambda_{1}^{k} v_{1} + \dots + \lambda_{n}^{k}v_{n}) \in G(V) 
  \end{equation}
is a polynomial function of the $\lambda_{i}$, of degree~$kd$ if $G$~has
degree~$d$.  By Lemma~\ref{thm:m2} this implies that $T_{V}$ is polynomial of
degree~$d$. 

Consider the following commutative diagram, in which $\Sigma:V^n\to V$ is
the sum map:
  \begin{equation}\label{eq:141}
  \begin{split}
      \xymatrix{ F(V^{n}) \ar[r]^-{F(\Sigma)}\ar[d]_{T_{V^{n}}} & F(V)
     \ar[d]^{T_{V}}\\ G(V^{n}) \ar[r]^-{G(\Sigma)} & G(V) .} 
  \end{split}
  \end{equation}
Let $j_{i}:V\to V^n$ be the inclusion of the $i^{\textnormal{th}}$~subspace
$0\oplus \cdots\oplus V\oplus \cdots\oplus 0\subset V^n$.  Under the top
arrow ~$F(\Sigma )$ the vector
  \begin{equation}\label{eq:142}
      \lambda_{1}^{k}F(j_{1})(v_{1})+\dots +
     \lambda_{n}^{k}F(j_{n})(v_{n}) \in F(V^{n}) 
  \end{equation}
maps to $ \lambda^{k}_1v_{1} + \dots +\lambda^{k}_nv_{n}\in F(V)$.  By
commutativity of the diagram, it therefore suffices to show that
  \begin{equation}\label{eq:143}
      T_{V^{n}}\bigl(\lambda_{1}^{k}F(j_{1})(v_{1})+\dots +
     \lambda_{n}^{k}F(j_{n})(v_{n})\bigr) \in G(V^n)
  \end{equation}
is a polynomial function of the $\lambda_{i}$ with coefficients in~$G(V^n)$.

Let $f:V^{n}\to V^{n}$ be the map
  \begin{equation}\label{eq:144}
      f= \lambda_{1}^{k/k_{1}}\epsilon_{1}+\dots +
     \lambda_{n}^{k/k_{n}}\epsilon_{n},
  \end{equation}
where $\epsilon _i(v_1,\dots ,v_n) = (0,\dots ,v_i,\dots ,0)$.  Since $G$ is
polynomial, it follows that that $G(f)$ is a polynomial in
$\lambda_{i}^{k/k_{i}}$ (and hence in the $\lambda_{i}$) with coefficients in
$\End\bigl(G(V^{n})\bigr)$.  Now follow the element
  \begin{equation}\label{eq:145}
      x = F(j_{1})(v_{1})+\dots + F(j_{n})(v_{n})\in F(V^{n}) 
  \end{equation}
around the commutative diagram
  \begin{equation}\label{eq:146}
  \begin{split}
      \xymatrix{ F(V^{n}) \ar[r]^{F(f)}\ar[d]_{T_{V^{n}}} & F(V^{n})
     \ar[d]^{T_{V^{n}}} \\ G(V^{n}) \ar[r]_{G(f)} & G(V^{n}), } 
  \end{split}
  \end{equation}
starting in the upper left corner.   It is sent by the top horizontal
arrow to  
  \begin{equation}\label{eq:147}
      \lambda_{1}^{k}F(j_{1})(v_{1}) + \dots +
     \lambda_{n}^{k}F(j_{n})(v_{n})\in F(V^n)
  \end{equation}
which in turn is sent by the right vertical arrow to~\eqref{eq:143}.  Under
the left vertical arrow $x$ is sent to $T_{V^{n}}(x)$ which, since $G(f)$ is
a polynomial endomorphism (in the $\lambda_{i}$), is sent by the bottom
horizontal arrow to a polynomial in the $\lambda_{i}$ with coefficients in
$G(V^{n})$.  If $G$~is polynomial of degree~$d$, then $G(f)$~is a polynomial
in $\lambda _1^{k/k_1},\dots ,\lambda _n^{k/k_n}$ of degree~$d$, so a
polynomial in $\lambda _1,\dots ,\lambda _n$ of degree~$kd$.  This completes
the proof.
\end{proof}

 \bigskip\bigskip
\providecommand{\bysame}{\leavevmode\hbox to3em{\hrulefill}\thinspace}
\providecommand{\MR}{\relax\ifhmode\unskip\space\fi MR }
\providecommand{\MRhref}[2]{%
  \href{http://www.ams.org/mathscinet-getitem?mr=#1}{#2}
}
\providecommand{\href}[2]{#2}

  \end{document}